\documentclass[11pt,reqno]{amsart}
\usepackage[hmargin=3cm,vmargin=3.3cm]{geometry}
\usepackage{amssymb}
\usepackage{color}
\usepackage{hyperref}
\usepackage{esint}
\newtheorem{theorem}{Theorem}
\newtheorem{proposition}{Proposition}
\newtheorem{corollary}{Corollary}
\newtheorem{lemma}{Lemma}

\theoremstyle{definition}
\newtheorem{remark}[lemma]{Remark}
\numberwithin{lemma}{section}
\numberwithin{claim}{section}
\numberwithin{equation}{section}
\numberwithin{proposition}{section}
\newcommand{\TR}{\textnormal{T}}
\newcommand{\ID}{\textnormal{Id}}
\newcommand{\RR}{\mathbb{R}}

\newcommand{\bl}{{\boldsymbol\ell}}

\newcommand{\bs}{\boldsymbol\sigma}
\newcommand{\by}{\boldsymbol y}
\newcommand{\bC}{\boldsymbol\chi}
\newcommand{\UN}{\mathbf 1}
\newcommand{\bM}{\mathbf M}
\newcommand{\bR}{\mathbf R}
\newcommand{\bW}{\mathbf W}
\newcommand{\bX}{\mathbf X}
\newcommand{\cB}{\mathcal B}
\newcommand{\cC}{\mathcal C}
\newcommand{\cD}{\mathcal D}
\newcommand{\cE}{\mathcal E}
\newcommand{\cH}{\mathcal H}
\newcommand{\cL}{\mathcal L}
\newcommand{\cM}{\mathcal M}
\newcommand{\cN}{\mathcal N}
\newcommand{\cS}{\mathcal S}
\newcommand{\ep}{\varepsilon}
\newcommand{\qk}{q_k^\alpha}
\newcommand{\BF}{\mathbf f}
\DeclareMathOperator{\DV}{div}
\DeclareMathOperator{\tr}{tr}

\begin{document}
\title[Multi-solitons for the critical wave equation]
{Existence of multi-solitons with any parameters\\
for the $5$D energy critical wave equation}
\author[Y.~Martel]{Yvan Martel}
\address{Laboratoire de mathématiques de Versailles, UVSQ, CNRS, Université Paris-Saclay, 
and Institut Universitaire de France, 45 avenue des États-Unis, 78035 Versailles Cedex, France}
\email{yvan.martel@uvsq.fr}
\author[F.~Merle]{Frank Merle}
\address{AGM - UMR 8088 - Analyse, Géométrie et Modélisation, CY Cergy Paris Université,
CNRS, 95302 Cergy-Pontoise, France 
et Institut des Hautes \'Etudes Scientifiques, France
}
\email{merle@ihes.fr}

\begin{abstract}
For the focusing, energy critical wave equation in dimension $5$,
we construct multi-solitons with any number  of solitons, any choice of signs, speeds, scaling parameters
and translation parameters.
This requires to revisit in depth previous constructions of multi-solitons
based on a unidirectional approach,
 to fully take into account the dimension of the space and the possibility for solitons to move in any direction.

Then, as a consequence of this more general construction and of the arguments developed in \cite{MaM2},
the inelastic nature of any collision of solitons is proved 
under a non-cancellation assumption on the parameters.
\end{abstract}

\maketitle

\section{Introduction}

We consider the focusing energy critical nonlinear wave equation in 5D
\begin{equation}\label{eq:NW}
\partial_t^2 u - \Delta u - |u|^{\frac 4{3}} u = 0, \quad (t,x)\in \RR\times \RR^5.
\end{equation}
Recall that the Cauchy problem for equation~\eqref{eq:NW} is locally well-posed
in the energy space $\dot H^1\times L^2$, using suitable Strichartz estimates. 
See \emph{e.g.}~\cite{KeMe} and references therein.
Note that equation~\eqref{eq:NW} is invariant by the $\dot H^1$ scaling:
if $u$ is a solution of~\eqref{eq:NW}, then for any $\lambda>0$, $u_\lambda$ defined by 
\[
u_\lambda(t,x)=\frac{1}{\lambda^{\frac 32}}u\left(\frac{t}{\lambda},\frac{x}{\lambda}\right),
\]
is also a solution of~\eqref{eq:NW} and satisfies $\|u_\lambda\|_{\dot H^1}=\|u\|_{\dot H^1}$.
The energy 
$E(u(t),\partial_t u(t))$ and the momentum $M(u(t),\partial_t u(t))$
of an $\dot H^1\times L^2$ solution are conserved, where
\[
E(u,v) = \frac 12 \int v^2 + \frac 12 \int |\nabla u|^2- \frac {3}{10} \int |u|^{\frac {10}{3}},
\quad
M(u,v)=\int v\nabla u.
\]
Recall also that the function $W$ defined by
\begin{equation*}
W(x) = \left( 1+ \frac {|x|^2}{15}\right)^{-\frac{3}2},\quad 
\Delta W + W^{\frac 73}=0 , \quad x\in \RR^5,
\end{equation*}
is a stationary solution of~\eqref{eq:NW}, called here ground state, or soliton.
By scaling, translation invariances and change of sign, we obtain a family of stationary solutions
of~\eqref{eq:NW} defined by 
$W_{\lambda,x_0,\pm}(x)=\pm \lambda^{-\frac 32}W\left(\lambda^{-1}(x-x_0)\right)$,
where $\lambda>0$ and $x_0\in \RR^5$.
Using the Lorentz transform, we obtain traveling waves, also called solitons.
Indeed, for any~$\bl\in \RR^5$, with $|\bl|< 1$, let
\begin{equation}\label{eq:Wl}
W_\bl(x)=W\biggl(x+\biggl(\frac{1}{\sqrt{1-|\bl|^2}}-1\biggr)\frac{\bl(\bl\cdot x)}{|\bl|^2}\biggr)
\quad \mbox{and}\quad w_\bl(t,x)=W_\bl(x-\bl t).
\end{equation}
Then, the functions $\pm w_{\bl}$, as well as rescaled
and translated versions of them, are solutions of~\eqref{eq:NW}.

\medbreak 
In this article, we prove a general existence result for multi-solitons of \eqref{eq:NW}.

\begin{theorem}[Multi-solitons with any set of parameters]\label{th:un}
Let $K\geq 2$. For all $k\in \{1,\ldots,K\}$, let $\lambda_k^\infty>0,$ ${\by}_k^\infty\in \RR^5$, $\epsilon_k\in \{\pm1\}$ and $\bl_k\in \RR^5$ with $|\bl_k|<1$
such that for any $k\neq m$, $\bl_k\neq \bl_m$.
Let
\begin{equation*}
W_k^\infty(t,x)= \frac {\epsilon_k}{(\lambda_k^\infty)^{\frac 32}} W_{\bl_k} \left( \frac{x - \bl_k t -{\by}^\infty_k}{\lambda_k^\infty}\right).
\end{equation*}
Then, there exist $T>0$ and a solution $u$ of~\eqref{eq:NW} on $[T,+\infty)$, satisfying
\begin{equation*}
\lim_{t\to +\infty}\biggl\|\nabla_{t,x} u(t) - \sum_{k=1}^K \nabla_{t,x} W_k^\infty(t)\biggr\|_{L^2(\RR^5)} = 0
.
\end{equation*}
\end{theorem}

Theorem \ref{th:un} means that pure asymptotic multi-solitons, in one sense of time, exist for any number of solitons and any choice of 
(two by two distinct) speeds, signs, scaling parameters and translation parameters.

\medskip

Compared to some other nonlinear models with solitons, an important difficulty to prove such a result for the critical wave equation \eqref{eq:NW} comes from the 
algebraic decay of the solitons as $|x|\to +\infty$, which yields nonlinear interactions of order $t^{-3}$
as $t\to +\infty$. For example, 
for the generalized Korteweg-de Vries, the nonlinear Schrödinger equations and the nonlinear Klein-Gordon equation
considered in \cite{CoMu,Ma05,MaM0}, the exponential decay of the solitons facilitates the construction
of multi-solitons.

As a consequence, the first existence result proved in \cite{MaM1}
for equation \eqref{eq:NW} was restricted
to the following configurations:
\begin{itemize}
	\item Any number of solitons, but with the condition that the speeds $\bl_k$ are all collinear to one fixed vector;
	\item Two solitons with any different speeds; this was obtained as a consequence of the fact
	that with only two solitons, one can use the Lorentz transformation to reduce to the previous case of collinear speeds.
\end{itemize}
These restrictions were due to the method of proof, based on a unidirectional approach, like the one
developed for the generalized Korteweg-de Vries equation in~\cite{Ma05,MaMT} for instance.

Another configuration was treated in \cite{Yu19}, proving the existence of multi-solitons
with any number of solitons with different speeds, under the following explicit
smallness condition on the speeds
\[
\sum_{j=1,\ldots,5} \max_{k=1,\ldots,K} \bl_{k,j}^2 \leq \frac 9{25}.
\]

In the present article, we revisit the construction technique, introducing
a new energy functional that fully takes into account the dimension of the space
and the possibility for solitons to move in any direction.
For the reader convenience, we provide in \S\ref{s:3.3} a heuristic presentation of the method.

\medbreak

As a consequence of the proof of the existence result,
we extend the inelasticity result \cite[Theorem 1]{MaM2}
to any configuration of multi-solitons, 
under a general non-vanishing assumption.
In \cite{MaM2}, the result was restricted to the configurations for which the construction was possible
(with a non-vanishing assumption).

\begin{corollary}[Inelasticity and dispersion as $t\to-\infty$]\label{th:dx}
Let $K\geq 2$. For all $k\in \{1,\ldots,K\}$, let $\lambda_k^\infty>0,$ ${\by}_k^\infty\in \RR^5$, $\epsilon_k\in \{\pm1\}$ and $\bl_k\in \RR^5$ with $|\bl_k|<1$ such that
for any $k\neq m$, $\bl_k\neq \bl_m$.
Assume that
\begin{equation}\label{eq:nv}
\sum_{k=1}^K (1-|\bl_k|^2)^{\frac 32}\lambda_k^\infty\Biggl(\sum_{m\neq k}\epsilon_{m}(\lambda_{m}^\infty)^{\frac 32}|\bs_{k,m}|^{-3}\Biggr)\neq 0
\end{equation}
where
\begin{equation}\label{eq:bb}
\bs_{k,m} = \bl_k-\bl_m+
\left(\frac 1{\sqrt{1-|\bl_m|^2}} - 1\right) \frac{\bl_m (\bl_m\cdot(\bl_k-\bl_m))}{|\bl_m|^2}.
\end{equation}
Then, there exists $C>0$ such that the solution $u$ constructed in Theorem \ref{th:un}
satisfies, for any $A>0$ sufficiently large,
\begin{equation}\label{eq:di}
\liminf_{t\to -\infty} \|\nabla u(t)\|_{L^2(|x|>|t|+A)} \geq {C A^{-\frac 52}}.
\end{equation}
\end{corollary}

The solution constructed in Theorem~\ref{th:un} exists at least on a time interval of the form $[T,+\infty)$ and it has an exact multi-soliton behavior as $t\to +\infty$.
In particular, it is ``dispersion-less'' at~$+\infty$.
Corollary \ref{th:dx} investigates the behavior of that particular solution in the other sense of time.
Note that the constructed solution does not necessarily exist for all $t<T$
(not all solutions of \eqref{eq:NW} are global).
However, by the finite speed of propagation and the small data Cauchy theory, it can be extended uniquely as a solution of~\eqref{eq:NW} for all $(t,x)\in \RR\times\RR^5$ in the region $|x|>|t|+A$, provided that $A$ is large enough (see section \ref{S:5} for a justification of this claim).

Corollary~\ref{th:dx} gives a lower bound on the amount of energy lost
as dispersion in this region as $t\to -\infty$. In particular, it implies that
the solution does not have a pure multi-soliton behavior as $t\to-\infty$
(even with different parameters).
This inelasticity result is proved using the 
strategy developed in \cite{MaM2}, based on the theory of channels of energy for the linear wave equation,
applied to the radial part of the solution.
We refer to \cite{DKM4} for the first introduction of this method 
and to~\cite{KeLS} for the exact result to be used in the present paper
(see also Proposition \ref{pr:ch}).
Here, the non-vanishing assumption on the parameters \eqref{eq:nv} is needed to avoid that,
at the main order, the sum of the signed, radial channeling components
due to each two by two collision of solitons vanishes.

\begin{remark}
The key arguments to prove the inelasticity are taken from \cite{MaM2}, and the improvement
of Corollary~\ref{th:dx} compared to \cite[Theorem 1]{MaM2} is due to the general construction method
of Theorem 1.
Recall that the proof of inelasticity requires a refined approximate solution (see Lemma \ref{le:43}),
so that the main order of the nonlinear interactions between the solitons is computed and then
justified on the multi-soliton solution. 
Indeed,
a more precise and technical result than Theorem \ref{th:un}, 
stated as Proposition \ref{pr:S1} below, is needed to prove Corollary~\ref{th:dx}.

If one does not intend to prove any inelasticity result, the construction method 
used in the proof of Theorem \ref{th:un} can be simplified, 
with no need of refined approximate solution. 
Such a simplified construction based on the energy method of the present paper
is expected to be accessible for any dimension $d\geq 6$.
The technique introduced in the present paper should also allow to improve other results
on the critical wave equation, such as \cite{Yu21,Yu22}.
\end{remark}

\begin{remark}
For references on the collision of solitons for other dispersive or wave models, 
including the case of integrable models, we refer to the introductions of \cite{MMan,MMin,MaM2}.
Following the discussion in \cite{MaM2}, one expects that understanding the global behavior of multi-solitons will enlighten the soliton resolution conjecture; see \cite{DKM4,DKM6,DKM7,DJKM,JeLa,JL25}
on this subject.

Apart from the integrability theory, there are few results dealing with the question of the collision of
solitons. 
In the case of 2-solitons with specific configurations of speeds, we refer to \cite{MMan,MMin,Pe11} for the generalized Korteweg-de Vries equation
and the one dimensional nonlinear Schrödinger equation.
For the energy critical wave and wave maps equations
with symmetry assumptions,
we refer to \cite{Je19,JeLa} proving, among other results, the existence and the classification (globally in time) of 2-soliton configurations.
\end{remark}

\subsection*{Notation}
In this paper, $i,j\in \{1,\ldots,5\}$ are related to the coordinates of the space $\RR^5$, while $k,m\in \{1,\ldots,K\}$ are related to soliton numbering.
Moreover, $\sum_k=\sum_{k=1}^K$ and $\sum_j = \sum_{j=1}^5$.
We will also denote $\partial_j = \partial/\partial x_j$.
We denote by $\UN_\Omega$ the indicator function of a subset $\Omega$ of $\RR^5$.
We denote $f(u) = |u|^\frac43 u$ and $f'(u)=\frac73 |u|^\frac43$.
Let
\[
\Lambda = \frac 32 + x \cdot \nabla,\quad 
\cL = -\Delta -\frac 73 W^\frac 43.
\]
We consider a radial $\cC^\infty$ function $\varphi:\RR^5\to[0,1]$, such that
$\varphi(x)=1$ for $|x|<1$ and $\varphi(x)=0$ for $|x|>2$.
Let $0< \alpha<1$ small to be chosen. Define
\[
\Theta (t,x) =
\begin{cases} 
1 & \quad \mbox{if $|x|<L t$}\\
\left( {Lt}/|x|\right)^{\alpha} & \quad \mbox{if $|x|\geq Lt$}
\end{cases}
\]
and
\[
\rho(t,x) = \begin{cases} 
t & \mbox{if $|x|<Lt$}\\
t^ \alpha \left(|x|/L\right)^{1- \alpha} & \mbox{if $|x|\geq L t$}
\end{cases}
\]
Observe that the functions $\Theta$ and $\rho$ are continuous
and piecewise $\cC^1$ on $I\times \RR^5$. Moreover, it holds on $\RR^5$,
$0\leq \Theta\leq 1$ and $\rho\geq t$.
We also introduce the notation
\[
\|f\|_\rho = \left(\int \rho f^2 \right)^\frac12.
\]
For $\bl\in \RR^5$ with $|\bl|<1$, we also define 
the operators 
\[
A_\bl=\partial_t +\bl \cdot \nabla,\quad 
\cL_\bl = -\Delta +\bl\cdot(\bl\cdot\nabla) -\frac 73 w_\bl^{\frac 43}.
\]
In the framework of Theorem \ref{th:un}, we define
\[
\bar \ell = \max (|\bl_1|,\ldots,|\bl_K|), \quad 0<\bar \ell <1,
\]
\begin{equation}\label{eq:sg}
\sigma = \frac 1{100} \min ( 1-\bar\ell, |\bl_k-\bl_m| ; k,m\in \{1,\ldots, K\}, k\neq m)>0,
\quad L= 1-\sigma.
\end{equation}
In the proof of Theorem~\ref{th:un}, we will use small parameters, like 
$\alpha>0$ and $\delta>0$ ($\delta$ will taken ultimately to be equal to
$\alpha/4$ and $\alpha$ will be taken small enough), and a large constant $T_0>1$ (to be chosen depending on $\delta$ and $\alpha$).
We will denote by $C$ a positive constant that may change from line to
line and that is independent of $\alpha$, $\delta$ and $T_0$
(such a constant may depend on the parameters of the solitons). The notation $A\lesssim B$ means that $A\leq C B$ for such a constant $C$. 
To indicate the dependence on a small parameter $\delta>0$, we will 
use the notation $A\lesssim_\delta B$.
Note that the parameter $\sigma>0$ has been fixed in \eqref{eq:sg}, but
it could have been taken arbitrarily small.

\section{Soliton interactions}

\subsection{Non homogeneous linearized equation}

Let $\bl\in \RR^5$ be such that $|\bl|<1$.
Let the functions $F$ and $G$ be defined by
\begin{equation}\label{eq:FG}
\begin{aligned}
&F = W^{\frac 43} + \kappa_\bl \Lambda W , \quad 
G = (1-\bl^2)^{-\frac 12} \kappa_\bl \left(\bl\cdot\nabla \Lambda W\right),\\
&\kappa_{\bl} = - (1-\bl^2) \frac{(W^{\frac 43},\Lambda W)}{\|\Lambda W\|_{L^2}^2}>0 .
\end{aligned}
\end{equation}
Set
\[
\zeta_\bl(t,x) = x-\bl t 
+ \biggl(\frac{1}{\sqrt{1-|\bl|^2}}-1\biggr) \frac{\bl(\bl\cdot (x-\bl t))}{|\bl|^2},
\]
and 
\[
f_\bl (t,x) = t^{-3} F ( \zeta_\bl(t,x)),\quad
g_\bl (t,x) = t^{-2} G ( \zeta_\bl(t,x) ).
\]
We reformulate Lemma 3.1 in \cite{MaM2} as follows.
\begin{lemma}\label{le:as}
There exists a smooth function $v_\bl:(t,x)\in [1,+\infty)\times \RR^5\mapsto v_\bl(t,x)$ such that, for all $0<\delta<1$, 
 for all $m\geq 0$,  $t\geq 1$, $x\in \RR^5$,
\begin{equation}\label{eq:Av}
\begin{aligned}
|A_\bl^m v_\bl(t,x)| \lesssim_\delta (t+\langle \zeta_\bl\rangle)^{-1} t^{-(1+m)} \langle \zeta_\bl\rangle^{-2+\delta}\\
|A_\bl^m\nabla v_\bl(t,x)| \lesssim_\delta (t+\langle \zeta_\bl\rangle)^{-1}t^{-(1+m)} \langle \zeta_\bl\rangle^{-3+\delta}\\
|A_\bl^m\nabla\nabla v_\bl(t,x)| \lesssim_\delta (t+\langle \zeta_\bl\rangle)^{-1}t^{-(1+m)} \langle \zeta_\bl\rangle^{-4+\delta}
\end{aligned}
\end{equation}
and
\begin{equation}\label{eq:AE}\begin{aligned}
|A_\bl^m \cE_{\bl}(t,x)|&\lesssim_\delta t^{-(4+m)+\delta} \langle \zeta_\bl\rangle^{-3}\\
|A_\bl^m\nabla \cE_{\bl}(t,x)|&\lesssim_\delta t^{-(4+m)+\delta} \langle \zeta_\bl \rangle^{-4}
\end{aligned}\end{equation}
where
\[
\cE_\bl=
\partial_t^2 v_\bl - \Delta v_\bl - \frac 73 w_\bl^{\frac 43} v_\bl 
- f_\bl - g_\bl.
\]
\end{lemma} 
\begin{remark}
Lemma \ref{le:as} means that asymptotically in large time, the function $v_\bl$ is an approximate solution of the problem $\cE_\bl=0$
at a given order of $t^{-1}$. This function, which is constructed explicitly by induction in \cite{MaM2}, will allow us to keep track of the nonlinear interaction between the solitons.
Indeed, the motivation for solving the specific non homogeneous problem 
of Lemma \ref{le:as} comes from the explicit interaction terms
$c_k t^{-3} |W_k|^{4/3}$ computed in Lemma \ref{le:in} below. 
The additional scaling terms in $F$ and $G$ involving the constant $\kappa_\ell$ are introduced for the problem
to have a relevant solution. This spectral constraint is due to the direction $\Lambda W$ of the kernel 
of the operator $\cL$.
\end{remark}

\subsection{Approximate multi-soliton}\label{S:2.2}

Let $C_0 > 1$ and $T_0> 1$ large to be fixed and $I\subset [T_0,+\infty)$ be an interval of $\RR$. 
We consider $\cC^1$ functions $\lambda_k>0$, $\by_k\in \RR^5$ defined on $I$.
We assume that these functions satisfy, for all $t\in I$,
\begin{equation}\label{eq:bs}
|\lambda_k(t)-\lambda^\infty_k|+|\by_k(t)-\by_k^\infty| \leq C_0 t^{-1}.
\end{equation}
Set
\[
\cD = \sum_k\left(|\dot{\lambda}_k|+|\dot{\by}_k|\right).
\]
For $\vec G= (G,H)$, define
\[(\theta_k G)(t,x) = \frac {\epsilon_k}{\lambda_k^{\frac 32}(t)} G\left(\frac {x -\bl_k t - \by_k(t)}{\lambda_k(t)}\right),
\quad
\vec \theta_k \vec G = \begin{pmatrix}\theta_k G \\[.2cm] \displaystyle \frac {\theta_k} {\lambda_k} H \end{pmatrix},\quad
\vec {\tilde \theta}_k \vec G = \begin{pmatrix}\displaystyle \frac {\theta_k}{\lambda_k} G\\[.4cm] \theta_k H \end{pmatrix}.
\]
In particular, set
\begin{equation}\label{eq:Wk}
W_k = \theta_k W_{\bl_k} ,\quad X_k=-\bl_k\cdot \nabla W_k,\quad 
\vec W_k = \begin{pmatrix} W_k \\[.2cm] \displaystyle X_k \end{pmatrix}.
\end{equation}
Let
\begin{equation}\label{eq:qk}
q_k(t,x) = (1+| x- \bl_k t - \by_k^\infty|^2)^{-\frac12},\quad
q =\sum_k q_k.
\end{equation}
Set (the constant $\sigma$ has been fixed in \eqref{eq:sg})
\[
B_k(t)=\left\{x \in \RR^5: |x-\bl_k t - \by_k^\infty|<\sigma t\right\},\quad 
 \Omega(t) = \RR^5 \setminus \cup_k B_k(t).
\]

\begin{lemma}\label{le:in}
For $m\neq k$, set
\[
\bs_{k,m} = 
\biggl(\frac 1{\sqrt{1-|\bl_m|^2}} - 1\biggr) \frac{\bl_m (\bl_m\cdot(\bl_k-\bl_m))}{|\bl_m|^2}
+\bl_k-\bl_m
\]
and
\[
c_{k} = \frac 73(15)^{\frac 32} \sum_{m\neq k}{\epsilon_m(\lambda_m^\infty)^{\frac 32}}{|\bs_{k,m}|^{-3}}.
\]
Then,
\[
f\Bigl(\sum_k W_k\Bigr)- \sum_{k} f(W_k) =
t^{-3}\sum_{k} c_k |W_k|^{\frac 43}+\bR_{\Sigma}
\]
where, for all $t\in I$, $x\in \RR^5$,
\[
 |\bR_{\Sigma}| \lesssim t^{-4} q^3,\quad
 |\nabla \bR_\Sigma|\lesssim t^{-4} q^4.
\]
\end{lemma}
\begin{proof}
Note that if $\bl_k\neq \bl_m$ then $\bs_{k,m}\neq 0$. Indeed, $\bs_{k,m}=0$
would imply that $\bl_k$ and $\bl_m$ are collinear, and in that case,
it is easy to compute $\bs_{k,m}$ to show that it is non zero.

We sketch the proof and refer to the proof of Lemma 4.1 in \cite{MaM2}.
First, for $x\in \Omega$, we have for all $k$,
\[
|W_k|^\frac73\lesssim q^7\quad\mbox{and}\quad  |W_k|^\frac43\lesssim q^4,
\]
and so for such $x$, $|\bR_\Sigma|\lesssim t^{-3}q^4\lesssim t^{-4} q^3$.
Second, for $x\in B_k$, 
using  $f(1+a)=\frac73 a +O( a^2)$, we obtain
\begin{align*}
f\biggl(\sum_m W_m\biggr) & = f(W_k) f\biggl( 1+\sum_{m\neq k} \frac{W_k}{W_m}\biggr)\\
& = f(W_k) \left(1+\frac 73 \frac{\sum_{m\neq k} W_m}{W_k}\right)
+ O\biggl(f(W_k)\sum_{m\neq k}\frac{W_m^2}{W_k^2}\biggr)
\end{align*}
so that
\[
f\biggl(\sum_m W_m\biggr)- f(W_k) 
= \frac 73  \biggl(\sum_{m\neq k} W_m \biggr) W_k^\frac43
+ O\biggl(\sum_{m\neq k} q_m^6q_k \biggr)
\]
Third, we easily check that for $x\in B_k$ and $m\neq k$,
\[
|W_m(t,x)-W_m(t,\bl_k t-\by_k)|\lesssim t^{-4} q_k^{-1}.
\]
Lastly,  the estimate $|W(x)-15^{3/2}|x|^{-3}|\lesssim |x|^{-5}$ for $|x|\gg 1$
shows that for $x\in B_k$ and $m\neq k$,
\begin{align*}
W_m(t,\bl_k t-\by_k)&=15^\frac32\epsilon_m\lambda_m^\frac32|\sigma_{k,m}|^{-3}
t^{-3} + O(t^{-5}).
\end{align*}
In particular, gathering the last two estimates, we obtain
for $x\in B_k$,
\[
\Bigl|W_m(t,x)
-15^\frac32\epsilon_m\lambda_m^\frac32|\sigma_{k,m}|^{-3}t^{-3} \Bigr|
W_k^\frac43 \lesssim t^{-4} q_k^3.
\]
This proves the pointwise estimate for $\bR_\Sigma$.
The estimate for $\nabla \bR_\Sigma$ is proved similarly, with a gain in space
decay.
\end{proof}
To cancel the main interaction terms $t^{-3} c_k |W_k|^{\frac 43}$ that appear in Lemma~\ref{le:in}, we define correction terms
to be added to the sum of solitons. Let
\begin{align*}
v_k(t,x) & = \frac 1{\lambda_k^3} v_{\bl_k}\left( \frac t{\lambda_k} , \frac {x-\by_k}{\lambda_k}\right)\\
z_k(t,x) & = \frac 1{\lambda_k^4} (\partial_tv_{\bl_k})\left( \frac t{\lambda_k} , \frac {x-\by_k}{\lambda_k}\right)
+ \frac{\kappa_{\bl_k}\epsilon_k}{2 \lambda_k^{\frac 12}t^2}\Lambda_k W_k(t,x)
\end{align*}
where the function $v_{\bl}$ is defined in Lemma \ref{le:as}
and $\kappa_\bl$ is defined in \eqref{eq:FG}.
Define
\[
a_k = - \frac {c_k\epsilon_k\kappa_{\bl_k}}2,\quad 
\vec v_k = \begin{pmatrix} v_k 
\\[.2cm]z_k \end{pmatrix}
\]
and
\begin{equation}\label{eq:WW}
\vec\bW =  \begin{pmatrix} \bW
\\[.2cm] \bX \end{pmatrix}
= \begin{pmatrix} \sum_k W_k + \sum_k c_k v_k
\\[.2cm] -\sum_k \bl_k \cdot \nabla W_k + c_k z_k\end{pmatrix}=
\sum_k \bigl( \vec W_k + c_k\vec v_k \bigr).
\end{equation}
The function $\vec\bW$ defined above is a refined approximate solution to the
multi-soliton problem in the sense that it takes into account the 
first order of the soliton interactions.

\begin{lemma}[{\cite[Lemma 4.3]{MaM2}}]\label{le:43}
Assume~\eqref{eq:bs}.
Then, the function $\vec\bW$ satisfies on $I\times \RR^5$ 
\[
\left\{\begin{aligned}
\partial_t\bW & = \bX - \bM_{\bW} - \bR_{\bW}\\
\partial_t	\bX & 
	= \Delta \bW +|\bW|^{\frac 43} \bW - \bM_{\bX}- \bR_{\bX}
\end{aligned}\right.
\]
where $\Lambda_k = \frac 32 + (x -\bl_k t -\by_k)\cdot \nabla$,
\begin{align*}
\bM_{\bW} &= \sum_k \Biggl( \frac{\dot\lambda_k}{\lambda_k} - \frac{a_k}{\lambda_k^{\frac 12}t^2} \Biggr) \Lambda_k W_k
+ \sum_k \dot {\by}_k \cdot \nabla W_k\\
\bM_{\bX} &= 
- \sum_k \Biggl( \frac{\dot\lambda_k}{\lambda_k} - \frac{a_k}{\lambda_k^{\frac 12}t^2} \Biggr) (\bl_k \cdot \nabla) \Lambda_k W_k 
- \sum_k ({\dot {\by}}_k \cdot \nabla) ( \bl_k \cdot \nabla) W_k
\end{align*}
and, for all $0<\delta< 1$, it holds on $I\times\RR^5$,
\begin{equation}\label{eq:bW}\begin{aligned}
&|\bW|
\lesssim_\delta q^{3}+t^{-1}q^{3-\delta} ,\quad
|\bX|\lesssim_\delta q^{4}+t^{-2}q^{3-\delta} ,\\
&|\nabla\bW - \sum_k \nabla W_k|
+|\bX+\sum_k \bl_k\cdot \nabla W_k|
\lesssim_\delta t^{-2} q^{3-\delta}.
\end{aligned}\end{equation}
Moreover, 
for all $0<\delta\leq \alpha/4$, it holds on $I\times\RR^5$,
\begin{equation}\label{eq:Rr}
\|\nabla\bR_\bW\|_\rho\lesssim_\delta t^{-\frac 32} \cD,\quad
\|\bR_\bX\|_\rho\lesssim_\delta t^{-\frac 72+\delta}+t^{-\frac 32} \cD.
\end{equation}
\end{lemma}
\begin{remark}
The main term $t^{-\frac72+\delta}$ in the estimate of $\|\bR_\bX\|_\rho$  
says that the function $\vec\bW$ is a better approximation 
of the multi-soliton problem than simply the sum $\sum_k \vec W_k$.
Note that getting the estimate \eqref{eq:Rr} is the point in the paper where we need 
$\alpha>0$ in the definition of the function $\rho$. Indeed, if $\rho(x)=|x|/L$
for $|x|>Lt$, then from the pointwise estimates on
$\nabla\bR_\bW$ and $\bR_\bX$, the quantities $\|\nabla\bR_\bW\|_\rho$ and $\|\bR_\bX\|_\rho$
are not well-defined. We actually need $0<\delta<\alpha/2$ to be able to
define these quantities.
\end{remark}
\begin{proof}
Note that using \eqref{eq:bs}, for $t$ large so that $|\by_k-\by_k^\infty|\leq \frac12$,
\[
|W_k|\lesssim (1+|x-\bl_k t-\by_k|)^{-3}
\lesssim (1+|x-\bl_k t-\by_k^\infty|-|\by_k-\by_k^\infty|)^{-3}\lesssim q_k^3
\]
and similarly, $|\nabla W_k|\lesssim q_k^4$.
Thus, obtain the estimates \eqref{eq:bW} directly from the definition of $\bW$ and $\bX$
and \eqref{eq:Av}.
Then, from the computations in the proof of Lemma 4.3 in \cite{MaM2}, we obtain
the system for $\vec \bW$ with
\begin{align*}
\bR_\bW & = \sum_k \biggl\{-3 \frac{\dot \lambda_k}{\lambda_k^4} v_{\bl_k} \left( \frac t{\lambda_k} , \frac {x-\by_k}{\lambda_k}\right)
- \frac{\dot \lambda_k}{\lambda_k^4} \frac{t}{\lambda_k}A_{\bl_k} v_{\bl_k} \left( \frac t{\lambda_k} , \frac {x-\by_k}{\lambda_k}\right) \\ 
&\quad-\frac{\dot \lambda_k}{\lambda_k^4}\left(\frac{x-\bl_k t-\by_k}{\lambda_k}\right)\cdot \nabla v_{\bl_k} \left( \frac t{\lambda_k} , \frac {x-\by_k}{\lambda_k}\right) - \frac{\dot {\by}_k}{\lambda_k^4} \cdot\nabla v_{\bl_k} \left( \frac t{\lambda_k} , \frac {x-\by_k}{\lambda_k}\right)\biggr\}
\end{align*}
and
\[
\bR_\bX = 
\bR_v + \bR_\Sigma +\bR_\cE+\bR_\cD
\]
where we define
\[
\bR_v = |\bW|^\frac43\bW - \Bigl|\sum_k W_k\Bigr|^\frac43\Bigl( \sum_k W_k\Bigr) - 
\frac 73 \sum_k c_k |W_k|^\frac 43 v_k,
\]
\[
\bR_\cE=
\sum_k \frac 1{\lambda_k^5} \cE_{\bl_k} \left( \frac t{\lambda_k} , \frac {x-\by_k}{\lambda_k}\right),
\]
and
\begin{align*}
\bR_\cD
& = \sum_k \biggl\{- \frac{\dot {\by}_k}{\lambda_k^5} \cdot \nabla \partial_t v_{\bl_k} \left( \frac t{\lambda_k} , \frac {x-\by_k}{\lambda_k}\right)-4 \frac{\dot \lambda_k}{\lambda_k^5} \partial_t v_{\bl_k}\left( \frac t{\lambda_k} , \frac {x-\by_k}{\lambda_k}\right) 
\\ &\quad 
 - \frac{\dot \lambda_k}{\lambda_k^5} \frac{t}{\lambda_k} A_{\bl_k} \partial_t v_{\bl_k}\left( \frac t{\lambda_k} , \frac {x-\by_k}{\lambda_k}\right) 
 - \frac{\dot \lambda_k}{\lambda_k^5} \left(\frac{x-\bl t-\by_k}{\lambda_k}\right)\cdot
\nabla \partial_t v_{\bl_k} \left( \frac t{\lambda_k} , \frac {x-\by_k}{\lambda_k}\right) 
\\&\quad -\frac {\kappa_{\bl_k}\epsilon_k}{2t^2\lambda_k^{\frac 12}} \frac{\dot\lambda_k}{\lambda_k}
\left(\frac 12 \Lambda_k W_k+\Lambda_k^2 W_k\right)
-\frac {\kappa_{\bl_k}\epsilon_k\bl_k}{2t^2\lambda_k^{\frac 12}} \dot{\by}_k\cdot\nabla\Lambda_k W_k\biggr\}.
\end{align*}
Now, we justify pointwise estimates for $\bR_\bW$ and $\bR_\bX$.
First, using \eqref{eq:Av}, we see easily 
\[
|\bR_\bW|\lesssim_\delta \cD t^{-2} q^{2-\delta},\quad
|\nabla \bR_\bW|\lesssim_\delta \cD t^{-2} q^{3-\delta}
\]
and similarly, 
\[
|\bR_\cD|\lesssim_\delta \cD t^{-2} q^{3-\delta},\quad
|\nabla\bR_\cD|\lesssim_\delta \cD t^{-2} q^{4-\delta}.
\]
Moreover, using \eqref{eq:AE}, 
\[
|\bR_\cE|\lesssim_\delta t^{-4+\delta} q^3,\quad
|\nabla \bR_\cE|\lesssim_\delta t^{-4+\delta} q^4.
\]
Then, we split $\bR_v = \bR_{v,1}+\bR_{v_2}$ where
\begin{align*}
\bR_{v_1} & = f\Bigr(\sum_k W_k+\sum_k v_k\Bigl) - f\Bigl(\sum_k W_k\Bigr) 
 - f'\Bigl(\sum_m W_m\Bigr) \Bigl( \sum_k c_k v_k \Bigr)\\
\bR_{v_2} &=  \sum_k c_kv_k \Bigl(f' \Bigl(\sum_m W_m\Bigr)  -  f'(W_k)\Bigr).
\end{align*}
Using \eqref{eq:Av}, we have
\[
|\bR_{v_1}|\lesssim 
\biggl(\sum_m |W_m|^\frac13\biggr) \biggl(\sum_k |v_k|^2\biggr)
\lesssim  t^{-4} q^{4},
\]
and similarly,
\[
|\nabla \bR_{v_1}|\lesssim 
\biggl(\sum_m |\nabla W_m||W_m|^{-\frac23}\biggr) \biggl(\sum_k |v_k|^2\biggr)+
\biggl(\sum_m |W_m|^\frac13\biggr) \biggl(\sum_k |\nabla v_k| |v_k|\biggr)
\lesssim t^{-4}q^{5}.
\]
and
\[
|\bR_{v_2}|
 \lesssim \sum_k |v_k| |W_k|\biggl| \biggl|1+ \frac{\sum_{m\neq k} W_m}{W_k}\biggr|^{\frac43}-1\biggr|
 \lesssim_\delta t^{-4+\delta} q^{4-\delta}.
\]
Similarly, we have
$|\nabla \bR_{v_2}|\lesssim_\delta t^{-4+\delta}q^{5-\delta}$.
Lastly, by Lemma \ref{le:in},
\[
|\bR_\Sigma|\lesssim_\delta t^{-4} q^{3},\quad
|\nabla \bR_\Sigma|\lesssim_\delta t^{-4}q^{4}.
\]
Therefore,
\[
|\bR_\bX|\lesssim_\delta t^{-4+\delta} q^3 + \cD t^{-2}q^{3-\delta},\quad
|\nabla\bR_\bX|\lesssim_\delta t^{-4+\delta} q^4 + \cD t^{-2}q^{4-\delta}.
\]
To deduce \eqref{eq:Rr} from the above pointwise estimates, we compute
$\|q^3\|_\rho$ and $\|q^{3-\delta}\|_\rho$.
First,
\begin{equation}\label{eq:QD}
\|q^3\|_\rho^2 = \int \rho q^6  dx 
\lesssim t \int_0^{Lt} (1+r)^{-2} dr + t^\alpha \int_{Lt}^{+\infty} r^{-1-\alpha} dr
\lesssim t + \alpha^{-1}\lesssim t,
\end{equation}
for $t$ sufficiently large (depending on $\alpha$).
For $0<\delta\leq \alpha/4$, we have similarly,
\begin{equation}\label{eq:qd}
\|q^{3-\delta}\|_\rho^2 = \int \rho q^{6-2\delta}  dx 
\lesssim t \int_0^{Lt} (1+r)^{-2+2\delta} dr 
+ t^\alpha \int_{Lt}^{+\infty} r^{-1-\alpha+2\delta} dr
\lesssim t + \alpha^{-1}\lesssim t,
\end{equation}
for $t$ sufficiently large (depending on $\alpha$).
\end{proof}

\subsection{Modulation}
\begin{lemma}[Properties of the decomposition]\label{le:dc} 
There exist $T_0\gg 1$ and $0<\omega_0\ll 1$ such that if $u(t)$ is a solution of~\eqref{eq:NW} which satisfies on $I$,
\begin{equation}\label{hyp:4}
\biggl\|\vec u - \sum_{k} \left( \vec W_k^\infty + c_k\vec V_k^\infty \right) \biggr\|_{\dot H^1\times L^2}< \omega_0,
\end{equation}
then there exist $\cC^1$ functions $\lambda_k>0$, $\by_k$ on $I$ such that, 
$\vec \ep(t)$ being defined by
\[
\vec \ep=\begin{pmatrix}\ep \\ \eta \end{pmatrix},\quad 
\vec u = \begin{pmatrix} u \\ \partial_t u \end{pmatrix} =
\vec \bW + \vec \ep, 
\]
the following hold on $I$.
\begin{itemize} 
\item\emph{Orthogonality and smallness.} For $j=1,\ldots,5$, 
\begin{equation}\label{eq:or}
(\ep,\Lambda_k W_k)_{\dot H_{\bl_k}^1}=
(\ep,\partial_j W_k)_{\dot H_{\bl_k}^1}= 0,
\end{equation}
\begin{equation}\label{bounds}
|\lambda_k -\lambda_k^{\infty}|
+|\by_k -\by_k^\infty|
+\|\vec \ep \|_{\dot H^1\times L^2}
\lesssim \biggl\|\vec u -\sum_{k} \left( \vec W_k^\infty + c_k\vec V_k^\infty \right) \biggr\|_{\dot H^1\times L^2}
\end{equation}
\item\emph{Equation of $\vec \ep$.} 
\[
\left\{\begin{aligned}
\partial_t \ep & = \eta + \bM_{\bW}+\bR_{\bW}\\
\partial_t \eta & 
= \Delta \ep +\left| \bW + \ep\right|^{\frac 43} (\bW + \ep)
- |\bW|^{\frac 43}\bW + \bM_{\bX}+ \bR_{\bX} .
\end{aligned}\right.
\]
\item\emph{Parameter estimates.}
\begin{equation}
\label{eq:ly}
\biggl|\frac {\dot \lambda_k}{\lambda_k} - \frac{a_k}{\lambda_k^{\frac 12}t^2}\biggr|
+|\dot {\by}_k|
\lesssim \left\|\vec\ep\right\|_{\dot H^1\times L^2} +t^{-4}.
\end{equation}
\item\emph{Unstable directions.} Let
$z_k^{\pm} = (\vec \ep ,\vec {\tilde\theta}_k \vec Z_{\bl_k}^{\pm})$.
Then, for any $0<\delta<1$, 
\begin{equation}
\label{eq:zk}
\left| \frac d{dt} z_k^{\pm} \mp \frac{\sqrt{\lambda_0}}{\lambda_k}(1-| \bl_k|^2)^{\frac 12}z_k^{\pm} \right|
\lesssim_\delta \| \vec\ep \|_{\dot H^1\times L^2}^2
+ t^{-1} \| \vec\ep \|_{\dot H^1\times L^2} + t^{-4+\delta}.
\end{equation}
\end{itemize}
\end{lemma}

\subsection{Bootstrap estimates}
We introduce another notation
\begin{equation}\label{eq:NN}
\cN = \sqrt{\|\nabla\ep\|_\rho^2 + \|\eta\|_\rho^2}.
\end{equation}
Let $\delta>0$ to be chosen sufficiently small.
We will assume the following bootstrap estimates, for some $C_0>0$ to be fixed,
\begin{equation}\label{eq:BS}
 |\lambda_k(t)-\lambda_k^\infty|\leq C_0 t^{-1},\quad
|\by_k(t)-\by_k^\infty|\leq  t^{-1},\quad
|z_k^{\pm}(t)| \leq t^{-\frac72},\quad
\cN(t)\leq t^{-\frac52+2\delta}.
\end{equation}
As a direct consequence of the bootstrap estimates, we have
\begin{equation}\label{eq:en}
\left\|\vec\ep\right\|_{\dot H^1\times L^2}
\lesssim t^{-\frac12} \cN
\lesssim t^{-3+2\delta}.
\end{equation}
From \eqref{eq:ly}, assuming \eqref{eq:BS}, we also have
\begin{equation}
\label{eq:pC}
\Biggl|\frac {\dot \lambda_k}{\lambda_k} - \frac{a_k}{\lambda_k^{\frac 12}t^2}\Biggr|+|\dot {\by}_k|
\lesssim t^{-\frac12} \cN + t^{-4} \lesssim t^{-3+2\delta},\quad
\cD\lesssim t^{-2},
\end{equation}
and from \eqref{eq:zk},
\begin{equation}
\label{eq:zC}
\left| \frac d{dt} z_k^{\pm} \mp \frac{\sqrt{\lambda_0}}{\lambda_k}(1-| \bl_k|^2)^{\frac 12}z_k^{\pm} \right|
\lesssim t^{-4+2\delta}.
\end{equation}
Moreover, from the definition of $\bM_\bW$, $\bM_\bX$ and \eqref{eq:pC},
we have
\begin{equation}\label{eq:Mm}
|\nabla \bM_\bW|+|\bM_\bX|\lesssim \left(t^{-\frac 12} \cN +t^{-4}\right) q^4
\lesssim t^{-3+2\delta}q^4,
\end{equation}
\begin{equation}\label{eq:MM}
\|\nabla \bM_\bW\|_{\rho}+\|\bM_\bX\|_{\rho}\lesssim \cN+t^{-\frac 72}.
\end{equation}
For future use, 
we derive additional estimates on $\ep$, reminiscent of the Hardy inequality.

\begin{lemma}
It holds \begin{equation}\label{eq:e1}
\int\rho|x|^{-2} \ep^2+ t \int q^2 \ep^2 + t^\alpha \int q^{1+\alpha}\ep^2
+\int \rho q^2 \ep^2
+ \int \frac{\Theta}{|x|} \ep^2 
+\int \Theta q \ep^2 
\lesssim \int \rho|\nabla\ep|^2.
\end{equation}
\end{lemma}
\begin{proof}
From the definitions of $\Theta$ and $\rho$, we recall that
\[
0\leq \Theta=\min(1,(Lt/|x|)^\alpha)\leq 1\quad\mbox{and}\quad
\rho = \max(t,t^\alpha(|x|/L)^{1-\alpha})=\Theta \max (t,|x|/L)\geq t.
\]
Recall the Hardy inequality, for $f\in \dot H^1(\RR^5)$,
\[
\int \frac{f^2}{|x|^2} \lesssim \int |\nabla f|^2,
\]
and more generally
\[
\int \frac{f^2}{|x|^{1+ \alpha}} \lesssim \int |x|^{1- \alpha}|\nabla f|^2,
\]
where the implicit constant can be taken independent of $\alpha$, for $\alpha\in[0,1]$.
From this, we have for example,
\[
\int q^2 \ep^2 \lesssim \int |\nabla\ep|^2 \lesssim t^{-1}\int \rho|\nabla\ep|^2.
\]
Moreover,
\[
\int \rho|x|^{-2}\ep^2
\lesssim t^\alpha \int (|x|^{-1-\alpha}+t^{1-\alpha}|x|^{-2})\ep^2
\lesssim t^{\alpha} \int |x|^{1-\alpha}|\nabla\ep|^2
+t \int |\nabla\ep|^2 \lesssim \int \rho|\nabla\ep|^2.
\]
Then, we have
\begin{align*}
\int q_k^{1+\alpha}\ep^2
 \lesssim \int |x-\bl_k t - \by_k|^{-1-\alpha} \ep^2
&\lesssim \int |x-\bl_k t - \by_k|^{1-\alpha} |\nabla\ep|^2\\
&\lesssim \int (|x|+t)^{1-\alpha}|\nabla\ep|^2\lesssim t^{-\alpha} \int \rho |\nabla\ep|^2.
\end{align*}
Thus, we have proved
\[
t^\alpha \int q^{1+\alpha}\ep^2\lesssim \int \rho |\nabla\ep|^2.
\]
The estimate $\int \rho q^2 \ep^2 \lesssim \int \rho |\nabla\ep|^2$ easily follows.
Now, setting $f=\Theta^{\frac12}\ep$, we have
\[
\nabla f = \Theta^\frac 12 \nabla \ep+\frac 12 (\nabla\Theta) \Theta^{-\frac12} \ep 
\]
and
\[
\int |x| |\nabla f|^2
=\int \Theta |x| |\nabla \ep|^2 + \int |x|(\nabla \Theta \cdot\nabla \ep) \ep
+ \int \frac{|\nabla \Theta|^2}{4\Theta} |x|\ep^2.
\]
Using $|x||\nabla \Theta|\lesssim \alpha \Theta$, we obtain
\[
\int |x| |\nabla f|^2 \lesssim \int \Theta |x| |\nabla \ep|^2
+ \alpha^2 \int \frac\Theta{|x|}\ep^2.
\]
Thus, as before,
\begin{align*}
\int \frac{\Theta}{|x|} \ep^2 
= \int \frac{f^2}{|x|} 
\lesssim \int |x||\nabla f|^2 \lesssim \int \Theta |x| |\nabla \ep|^2
+ \alpha^2 \int \frac\Theta{|x|}\ep^2
\end{align*}
and so for $\alpha>0$ small,
\[
\int \frac{\Theta}{|x|} \ep^2 
\lesssim \int \Theta |x| |\nabla \ep|^2\lesssim \int \rho |\nabla \ep|^2.
\]
We also check that
\[
\int \Theta q_k \ep^2 
\lesssim \int \Theta |x-\bl t - \by_k| |\nabla \ep|^2
\lesssim \int \Theta (|x|+|t|) |\nabla \ep|^2
\lesssim \int \rho |\nabla \ep|^2.
\]
\end{proof}

\section{Energy estimates}\label{S:3}

\subsection{Technical estimates}
We set
\[
\Phi = \Theta \psi,
\quad 
\psi=\prod_k (1-\qk).
\]
We also define (the functions $q_k$ are defined in \eqref{eq:qk})
\begin{equation}\label{eq:dx}
\bC(t,x) = \Bigl(x- \sum_k \left((x-\bl_k t - \by_k^\infty) \qk(t,x) 
+\BF_k(t) \varphi_k(t,x)\right)\Bigr)\Theta(t,x)
\end{equation}
where
\begin{equation}\label{eq:vv}
\varphi_k(t,x)
=\varphi\bigl(t^{-1+\frac\alpha2}\left(x-\bl_k t-\by_k^\infty\right)\bigr)
\end{equation}
and
\begin{equation}\label{eq:fk}
\BF_k(t)= 
\by_k^\infty-\sum_{m\neq k}\left((\bl_k-\bl_m)t+(\by_k^\infty-\by_m^\infty)\right)
q_m^\alpha(t,\bl_k t+\by_k^\infty).
\end{equation}
Note that this choice of $\BF_k$ implies that 
$\bC(t,\bl_k t+\by_k^\infty)=\bl_k t$ and
$|\BF_k|\lesssim t^{1-\alpha}$.
Now, we give some estimates on the above defined functions.
\begin{lemma}
\begin{itemize}
\item\emph{Estimates on $\qk$.}
\begin{equation}\label{eq:pk}
|\partial_j \qk|+|\partial_t \qk|\lesssim \alpha q_k^{1+\alpha},\quad
|\partial_j^2\qk|
\lesssim \alpha q_k^{2+\alpha}.
\end{equation}
\item\emph{Estimates on $\Phi$.}
\begin{equation}\label{eq:Ph}
|\partial_j \Phi|\lesssim \alpha \Theta q,\quad
|\partial_t \Phi| \lesssim \alpha\Theta \left(t^{-1}+q\right).
\end{equation}
\item\emph{Estimates on $\bC$.} 
\begin{equation}\label{eq:GG}
\mbox{For $i\neq j$,}\quad
|\partial_j \bC_j - \Phi|+|\partial_i \bC_j|\lesssim \alpha \Theta,\quad 
\Bigl|\partial_t \bC -\sum_k \bl_k \qk \Bigr|
\lesssim \alpha t^{-1} \rho.
\end{equation}
\item\emph{Estimate on $\bC$ near the solitons.}
\begin{equation}\label{eq:cl}
\mbox{For $x\in B_k$,}\quad
|\bC -\bl_k t |\lesssim q_k^{-1} \min(\alpha q_k^{-2},1).
\end{equation}
\end{itemize}
\end{lemma}
\begin{proof}
Note that
\begin{align*}
\partial_j \qk(t,x)
& = - \alpha (x_j - \bl_{k,j} t - \bl_{k,j}^\infty) 
 q_k^{2+\alpha}(t,x), \\
\partial_t q_k^\alpha(t,x)
&=-\bl_k\cdot\nabla q_k^\alpha(t,x),
\\
\partial_j^2\qk(t,x)
&= - \alpha 
q_k^{2+\alpha}(t,x)
 + \alpha (2+ \alpha) (x_j - \bl_{k,j} t - \bl_{k,j}^\infty)^2
q_k^{4+\alpha}(t,x),
\end{align*}
which implies \eqref{eq:pk}.
Then, we see that
\[
|\nabla \Phi|\lesssim
|\nabla\Theta|\psi+\Theta|\nabla\psi|
\lesssim \alpha \Theta ( |x|^{-1} \UN_{|x|>Lt} +q^{1+\alpha})
\lesssim \alpha \Theta q.
\]
Note that $\partial_t \Theta = \alpha t^{-1} \Theta \UN_{|x|>Lt}$ and thus
\begin{align*}
\partial_t \Phi
&= \alpha t^{-1} \Theta \UN_{|x|>Lt} \psi
- \Theta\sum_k  (\partial_t \qk)\prod_{m\neq k} (1-q_m^\alpha)\\
&=\alpha \Theta \Bigl( t^{-1} \UN_{|x|>Lt}\psi+ 2 \sum_k (\bl_k \cdot \nabla q_k) q_k^{\alpha-1}\prod_{m\neq k} (1-q_m^\alpha)
\Bigr).
\end{align*}
This implies that $|\partial_t \Phi| \lesssim \alpha\Theta (t^{-1}+q)$
and \eqref{eq:Ph} is proved.

We prove \eqref{eq:GG}. First, for $i=j$, one has
\begin{align*}
\partial_j\bC_j& = \Bigl(1-\sum_k q_k^\alpha-\sum_k(x_j-\bl_{k,j} t- \by_{k,j}^\infty) \partial_j \qk - \BF_{k,j}\partial_j \varphi_k\Bigr) \Theta\\
&\quad + \Bigl(x_j- \sum_k \left((x-\bl_{k,j} t - \by_{k,j}^\infty) \qk 
+\BF_{k,j} \varphi_k \right)\Bigr)\partial_j\Theta.
\end{align*}
By the definition of $\psi$, it is easy to see that
$|1-\sum_k \qk-\psi|\lesssim t^{-\alpha}$.
By \eqref{eq:pk},
\[
|x_j - \bl_{k,j}t-\by_{k,j}^\alpha| \, |\partial_j \qk|
\lesssim \alpha \qk \lesssim \alpha.
\]
Then, by the definition of $\varphi_k$, and $|\BF_{k,j}|\lesssim t^{1-\alpha}$
(from its definition in \eqref{eq:fk}),
\[
|\BF_{k,j}\partial_j\varphi_k|\lesssim t^{1-\alpha}t^{-1+\frac\alpha2}
\lesssim t^{-\frac\alpha2}.
\]
We also check that $\partial_j\Theta=-\alpha\frac{x_j}{|x|^2}\Theta\UN_{|x|>LT}$,
and so
\[
\Bigl|x- \sum_k \left((x-\bl_k t - \by_k^\infty) \qk 
+\BF_k \varphi_k \right)\Bigr|\, |\partial_j\Theta|
\lesssim \alpha \frac{|x|+t}{|x|}\UN_{|x|>Lt} \lesssim \alpha.
\]
Therefore, for $T_0$ sufficiently large, we have proved
$|\partial_j\bC_j-\psi|\lesssim\alpha$.
Second, for $i\neq j$, we have
\begin{align*}
\partial_i\bC_j& = \Bigl(-\sum_k(x_j-\bl_{k,j} t- \by_{k,j}^\infty) \partial_i \qk - \BF_{k,j}\partial_i \varphi_k\Bigr) \Theta\\
&\quad + \Bigl(x_j- \sum_k \left((x-\bl_{k,j} t - \by_{k,j}^\infty) \qk 
+\BF_{k,j} \varphi_k \right)\Bigr)\partial_i\Theta.
\end{align*}
By the same estimates as before, we obtain $|\partial_i\bC_j|\lesssim\alpha$,
for $i\neq j$, which proves \eqref{eq:GG}, for the space derivatives.
For the time derivative, we compute
\begin{align*}
\partial_t \bC & = 
\Bigl(\sum_k \left(\bl_k \qk+(x-\bl_k t-\by_k^\infty)\partial_t \qk +\BF_k'\varphi_k+\BF_k \partial_t \varphi_k\right)\Theta\\
&\quad + \Bigl(x- \sum_k \left((x-\bl_k t - \by_k^\infty) \qk 
+\BF_k \varphi_k \right)\Bigr)\partial_t\Theta.
\end{align*}
By \eqref{eq:pk}, we have
\[
|x-\bl_k t-\by_k^\infty|\,|\partial_t \qk|
\lesssim \alpha \Theta \qk \lesssim \alpha.
\]
By \eqref{eq:fk}, $|\BF_k'\varphi_k|\Theta \lesssim t^{-\alpha}$.
By a direct computation from the definition of $\varphi_k$,
\begin{align*}
\partial_t\varphi_k& = -\Bigl(1-\frac\alpha2\Bigr)t^{-1}
t^{-1+\frac\alpha2}(x-\bl_kt-\by_k^\infty)
\cdot\nabla\varphi\big(t^{-1+\frac\alpha2}(x-\bl_kt-\by_k^\infty)\big)\\
&\quad - t^{-1+\frac\alpha2}
\bl_k\cdot\nabla \varphi\big(t^{-1+\frac\alpha2}(x-\bl_kt-\by_k^\infty)\big).
\end{align*}
From this identity, we obtain $|\partial_t\varphi_k|\lesssim t^{-1+\frac\alpha2}$,
and using also $|\BF_k|\lesssim t^{1-\alpha}$ from its definition, 
we obtain $|\BF_k\partial_t\varphi_k|\lesssim t^{-\frac\alpha2}$.
By the definition of $\Theta$, 
$\partial_t \Theta =   \alpha t^{-1}\Theta\UN_{|x|>Lt}$, which implies that
$|\partial_t \Theta|\lesssim  \alpha t^{-1}\Theta$ and thus
\[
 \Bigl|x- \sum_k \left((x-\bl_k t - \by_k^\infty) \qk 
+\BF_k \varphi_k \right)\Bigr|\, |\partial_t\Theta|
\lesssim (|x|+t) \alpha t^{-1}  \Theta\lesssim \alpha t^{-1} \rho.
\]
This finishes the proof of \eqref{eq:GG}.

Lastly, we prove \eqref{eq:cl}.
We check that $\bC(t,\bl_k t+\by_k^\infty)=\bl_k t$.
Indeed, $\Theta(\bl_kt +\by_k^\infty)=1$,
$\varphi_k(\bl_k t+\by_k^\infty)=\varphi(0)=1$ and $\varphi_m(\bl_k t+\by_k^\infty)=0$
for $m\neq k$. Thus,
\[
\bC(t,\bl_kt+\by_k^\infty)=
\bl_kt+\by_k^\infty -\BF_k 
-\sum_{m\neq k}\left((\bl_k-\bl_m)t+(\by_k^\infty-\by_m^\infty)\right)
q_m^\alpha(t,\bl_k t+\by_k^\infty)
\]
which implies the result by the definition of $\BF_k$
in \eqref{eq:fk}.
Since $\Phi\leq 1$, by \eqref{eq:GG}, one has
$|\nabla\bC|\lesssim 1$ and so, by the mean value Theorem,
$|\bC-\bl_k t|\lesssim |x-\bl_k t-\by_k^\infty|\lesssim q_k^{-1}$
holds on $\RR^5$.
Moreover, on $B_k$, 
by the general inequality $0\leq 1- (1+a^2)^{-\frac\alpha2}
\lesssim \alpha a^2$, we have
\[
\Phi\leq |1-\qk|\Theta \lesssim \alpha |x-\bl_kt-\by_k^\infty|^2.
\]
Thus, also using \eqref{eq:GG}, we have on $B_k$,
\[
|\nabla\bC|\lesssim \alpha (1+|x-\bl_k t -\by_k^\infty|^2).
\]
By the mean value Theorem, we obtain on $B_k$,
\[
|\bC(t,x)-\bl_kt|
\lesssim \alpha|x-\bl_k t -\by_k^\infty| (1+|x-\bl_kt-\by_k^\infty|^2)
\lesssim \alpha q_k^{-3}.
\]
This finishes the proof of \eqref{eq:cl}.
\end{proof}

\subsection{Energy functional}
We define
\[
\cH = \cH_1+\cH_2+\cH_3
\] 
where
\[
\cH_1 = \int \rho \left(|\nabla\ep|^2+\eta^2- 2 (F (\bW +\ep )-F (\bW )-f (\bW )\ep ) \right),
\]
\[
\cH_2 = 2 \int (\bC\cdot \nabla \ep ) \eta,\quad
\cH_3 = 4 \int \Phi \ep \eta.
\]
\begin{remark}
The first part $\cH_1$ of the energy functional has a classical expression for a linearized
wave problem around an approximate solution $\bW$.
The second part $\cH_2$, where $\bC\approx x$ far from the solitons, but
$\bC\approx \bl_k t$ close to the soliton $W_k$ allows to have an energy functional 
which is adapted to each soliton $W_k$.
The third part is algebraically important to compensate the presence of the term $\cH_2$
far from the solitons
when differenting $\cH$.
The combinaison of the two parts $\cH_2$ and $\cH_3$ in this energy functional is one of the
key points of this paper compared to \cite{MaM1} and \cite{MaM2}.
It allows to take into account the dimension of the space, while the functional
used in \cite{MaM1} and \cite{MaM2} are inherited from $1$D constructions.
\end{remark}

\subsection{Heuristics}\label{s:3.3}

Formally, the functional $\cH$ has different expressions according to the following three space-time regions:
\begin{itemize}
\item Outside the wave cone: 
\[
 \Omega_{\rm out} = \{(t,x)\in \RR\times\RR^5 : |x|\geq t\}.
\]
\item Inside the wave cone, far from the solitons
\[
 \Omega_{\rm inter} = \{(t,x)\in \RR\times\RR^5 : |x|<t, \ |x-\bl_k t-\by_k^\infty|\gg 1\}.
\]
\item Close to the solitons:
\[
B_k = \{(t,x) : |x - \bl_k t -\by_k^\infty|\ll t\}
\]
\end{itemize}
For these heuristics, we systematically neglect the soliton $W_k$ outside $B_k$ and
all the nonlinear terms in $\ep$.
We also discard boundary terms while integrating by parts, because they compensate 
when gathering the computations.
Close to the soliton $k$, i.e., in the ball $B_k$, we take $q_k\equiv 1$ and far from the soliton $k$, we take
 $q_k\equiv 0$.
Moreover, we formally take $ \alpha=0$ in the definition of the functions
$\Theta$ and $\rho$.

We formally justify a coercivity property for $\cH$
and a differential inequality on $\cH$ of the form $\cH'+\frac2t\cH\geq0$,
up to error terms that are ignored

\medskip

\emph{Outside the wave cone}, we have the corresponding part of the functional
\[
\cH_{\rm out} 
=\int_{\Omega_{\rm out}} |x|\left(|\nabla\ep|^2+\eta^2\right)
+ 2 (x\cdot\nabla\ep)\eta + 4\int_{\Omega_{\rm out}} \ep\eta.
\]
For the positivity of $\cH_{\rm out}$, we proceed as follows
\begin{align*}
\cH_{\rm out} 
&\geq \int_{\Omega_{\rm out}} |x|\left(\frac{x}{|x|}\cdot\nabla\ep + \eta\right)^2
+ 4\int_{\Omega_{\rm out}} \ep\eta\\
&\geq\int_{\Omega_{\rm out}} |x|\left(\frac{x}{|x|}\cdot\nabla\ep + \eta\right)^2
+ 4\int_{\Omega_{\rm out}} \ep\left(\frac{x}{|x|}\cdot\nabla\ep + \eta\right)
- 4\int_{\Omega_{\rm out}} \ep\left(\frac{x}{|x|}\cdot\nabla\ep \right)\\
&\geq\int_{\Omega_{\rm out}} |x|\left(\frac{x}{|x|}\cdot\nabla\ep + \eta +2\frac{\ep}{|x|}\right)^2
-4\int_{\Omega_{\rm out}}\frac{\ep^2}{|x|}
+ 2 \int \ep^2 \DV\left(\frac{x}{|x|}\right).
\end{align*}
Since $\DV (\frac{x}{|x|} ) = \frac4{|x|}$, we obtain
\[
\cH_{\rm out} 
\geq\int_{\Omega_{\rm out}} |x|\left(\frac{x}{|x|}\cdot\nabla\ep + \eta +2\frac{\ep}{|x|}\right)^2
+4\int_{\Omega_{\rm out}}\frac{\ep^2}{|x|}.
\]
Moreover, in the region $ \Omega_{\rm out}$, the equation reduces formally to
\[
\partial_t \ep = \eta, \quad \partial_t \eta = \Delta \ep.
\]
Thus,
\begin{align*}
\frac d{dt} \cH_{\rm out}
& = 2 \int_{\Omega_{\rm out}} |x|\left(\nabla\partial_t\ep\cdot \nabla \ep +\eta\partial_t \eta \right)
+ 2\int (x\cdot\nabla\partial_t\ep)\eta
+ 2 \int(x\cdot\nabla\ep)\partial_t\eta \\
&\quad + 4\int_{\Omega_{\rm out}} \left(\eta\partial_t\ep+ \ep\partial_t\eta\right)\\
& = 2 \int_{\Omega_{\rm out}} |x|\left(\nabla\eta \cdot\nabla \ep +\eta\Delta \ep
+ 2 (x\cdot\nabla\eta)\eta
+ 2 (x\cdot\nabla\ep)\Delta\ep\right)  + 4\int_{\Omega_{\rm out}} \left(\eta^2+ \ep\Delta\ep\right)\\
&=- \int_{\Omega_{\rm out}} |\nabla\ep|^2 + \eta^2 + 2 \biggl(\frac{x}{|x|}\cdot \nabla \ep\biggr) \eta.
\end{align*}
Since
\[
|\nabla\ep|^2 + \eta^2 + 2 \biggl(\frac{x}{|x|}\cdot \nabla \ep\biggr) \eta
\geq \biggl( \frac{x}{|x|}\cdot\nabla\ep+\eta\biggr)^2 \geq 0,
\]
we have
\[
 \int_{\Omega_{\rm out}} |\nabla\ep|^2 + \eta^2 + 2 \biggl(\frac{x}{|x|}\cdot \nabla \ep\biggr) \eta 
 \leq \frac 1t \int_{\Omega_{\rm out}} |x| \left(|\nabla\ep|^2+\eta^2\right)
+ 2 (x\cdot \nabla\ep)\eta 
\]
and
\begin{align*}
\frac d{dt} \cH_{\rm out}+\frac2t\cH_{\rm out}
&\geq\frac 1t \int_{\Omega_{\rm out}} |x| \biggl( \frac{x}{|x|}\cdot\nabla\ep+\eta\biggr)^2
+ \frac 8t\int_{\Omega_{\rm out}}\ep\eta.
\end{align*}
As before,
\begin{align*}
&\int_{\Omega_{\rm out}} \frac{|x|}{t} \biggl( \frac{x}{|x|}\cdot\nabla\ep+\eta\biggr)^2
+\frac 8t\int_{\Omega_{\rm out}}\ep\eta\\
&\quad = \int_{\Omega_{\rm out}} \frac{|x|}{t} \biggl( \frac{x}{|x|}\cdot\nabla\ep+\eta\biggr)^2
+ \frac 8t\int_{\Omega_{\rm out}}\ep\biggl(\frac{x}{|x|}\cdot\nabla\ep+\eta\biggr)
- \frac 4t\int_{\Omega_{\rm out}}\ep^2\DV\biggl(\frac{x}{|x|}\biggr)\\
&\quad \geq \int_{\Omega_{\rm out}} \frac{|x|}{t} \biggl( \frac{x}{|x|}\cdot\nabla\ep+\eta
+\frac 4{|x|} \ep\biggr)^2.
\end{align*}
Observe that the cancellation between
the terms $-16\int_{\Omega_{\rm out}} {\ep^2}/{|x|}$ coming from completing the square
and 
$- \frac 4t\int_{\Omega_{\rm out}}\ep^2\DV ( {x}/{|x|} )
=\frac{16}t \int_{\Omega_{\rm out}}{\ep^2}/{|x|}$.

Therefore, in the region $\Omega_{\rm out}$, we have obtained formally
\[
\frac d{dt} \cH_{\rm out}+\frac2t\cH_{\rm out}\geq0.
\]

\medskip

\emph{Inside the wave cone but far from the solitons}, we have
\begin{align*}
\cH_{\rm inter} 
& = t\int_{\Omega_{\rm inter}} \left(|\nabla\ep|^2+\eta^2\right)
+ 2 \int_{\Omega_{\rm inter}} (x\cdot\nabla\ep)\eta
+4\int_{\Omega_{\rm inter}}\ep\eta\\
& \geq\int_{\Omega_{\rm inter}} t \biggl(\frac{x}t\cdot\nabla\ep+\eta\biggr)^2
+ t\int_{\Omega_{\rm inter}} \biggl(1-\frac{|x|^2}{t^2}\biggr)|\nabla\ep|^2\\
&\quad +4 \int_{\Omega_{\rm inter}} \ep \biggl(\eta+\frac xt\cdot\nabla\ep\biggr)
-4\int_{\Omega_{\rm inter}} \ep\biggl(\frac{x}t\cdot\nabla\ep\biggr)\\
&\geq \int_{\Omega_{\rm inter}} t\biggl(\frac{x}t\cdot\nabla\ep+\eta+2\frac\ep t\biggr)^2
+ t\int_{\Omega_{\rm inter}} \biggl(1-\frac{|x|^2}{t^2}\biggr)|\nabla\ep|^2
+16\int_{\Omega_{\rm inter}} \frac{\ep^2}{t}.
\end{align*}
Moreover, in the region $ \Omega_{\rm inter}$, the equation also reduces formally to
\[
\partial_t \ep = \eta, \quad \partial_t \eta = \Delta \ep
\]
and so
\begin{align*}
\frac d{dt} \cH_{\rm inter} & = \int_{\Omega_{\rm inter}} \left(|\nabla\ep|^2+\eta^2\right)
+ 2 t\int_{\Omega_{\rm inter}} \left(\nabla\eta\cdot\nabla\ep+\eta\Delta\ep\right)\\
&\quad
+ 2 \int_{\Omega_{\rm inter}} (x\cdot\nabla\eta)\eta+2 \int_{\Omega_{\rm inter}}(x\cdot\nabla\ep)\Delta\ep
+ 4 \int_{\Omega_{\rm inter}} (\eta^2 +\ep\Delta\ep)= 0.
\end{align*}

\medskip

\emph{Lastly, in the vicinity of soliton $W_k$,} we have
\[
\cH_{B_k} = t \int_{B_k} \left( |\nabla\ep|^2
+\eta^2+ 2 (\bl_k\cdot\nabla\ep)\eta
-\frac 73 W_k^\frac43\ep^2\right).
\]
The coercivity of this functional, up to the two exponential instable directions (backward and forward),
and the non-trivial kernel,
is well-known (see for example \cite[Lemma 2.2]{MaM2}).
Moreover, the equation of $(\ep, \eta)$ in $B_k$ is formally reduced to
\[
\partial_t \ep = \eta, \quad 
\partial_t \eta = \Delta \ep + \frac 73 W_k^\frac43 \ep
\]
(we discard modulation terms, that will vanish at the main order, and nonlinear terms).
Using this equation and then integrating by parts, we check that 
\begin{align*}
\frac{d}{dt}\cH_{B_k}
& =\frac 1t\cH_{B_k}
+2 t\int_{B_k} \left(\nabla \eta\cdot\nabla\ep
+\eta\Bigl(\Delta\ep+\frac 73W_k^\frac43\ep\Bigr)\right)
\\
&\quad 
+2t\int_{B_k} \left((\bl_k\cdot\nabla\eta)\eta + (\bl_k\cdot\nabla\ep)\Bigl(\Delta\ep+
\frac73 W_k^\frac73\ep\Bigr)
-\frac 73 W_k^\frac43\ep \eta \right)\\
&\quad + t \frac 73\frac43\int_{B_k} (\bl_k\cdot\nabla W_k) W_k^\frac13 \ep^2\\
& =\frac 1t\cH_{B_k}\geq 0.
\end{align*}

In the next two subsections, we give a rigorous justification of the above heuristics.

\begin{remark}
The different cut-off functions $\rho$, $\Theta$, $\bC$ are introduced for various reasons.
First, it is important to split $\RR\times\RR^5$ into the regions $|x|>Lt$ and $|x|<Lt$
(instead of exactly $|x|>t$ and $|x|<t$) to obtain the coercivity of $\cH$.
Second, the decay of the approximate multi-soliton, as described in Proposition \ref{le:as}
is not sufficient to use the more simple weight function $\rho=|x|/L$ for $|x|>Lt$
instead of the one defined in the notation section.
Third, the function $\bC$ looks like $x$, which allows to connect the functionals adapted to each of the solitons.
Lastly, the function $\qk$ used to localized around the soliton $W_k$ has a weak decay for
$\alpha>0$ small, but it seems that we cannot use a stronger localization, because of
error terms coming from derivatives of this localization function.
Here, error terms will be controled using the fact that $\alpha$ is small.
\end{remark}

\subsection{Coercivity of the energy}
\begin{lemma}\label{le:co}
There exists $\mu>0$ such that
\[
\cH
\geq \mu \cN^2-\frac 1{\mu} \sum_k \left((z_k^-)^2+ (z_k^+)^2\right).
\]
\end{lemma}
\begin{proof}
\emph{Decomposition of $\cH$.}
The decomposition of $\cH$ into three pieces $\cH_1+\cH_2+\cH_3$
in its definition is well-adapted to compute its time derivative
using the equations of $(\ep, \eta)$.
To prove the coercivity property, 
we need to decompose $\cH$ in a different way, taking into account the various regions
(close to solitons and far from the solitons).
We set
\[
\cH = \sum_{j=1}^7 h_j
\]
where the terms $h_j$ are defined as described below.
\begin{itemize}
\item The first term 
\[
h_1 =
\int \rho (1-\psi)\left(|\nabla\ep|^2+\eta^2\right) 
+ \sum_k \int \rho 
\left( 2 \qk(\bl_k\cdot \nabla \ep ) \eta 
- f'(W_k) \ep^2\right)
\]
concerns the regions close to the various solitons,
since the functions $1-\psi$, $\qk$ and $f'(W_k)$ are localized around the solitons.
To estimate this term we will use a localized variant of a standard positivity property of the 
linearized operator $\cL$ (under orthogonality conditions).
\item The second term 
\[
h_2 = \frac 1L\int \Phi |x| (|\nabla\ep|^2+\eta^2)
+2 \int \Phi(x\cdot \nabla \ep) \eta +4 \int \Phi \ep \eta 
\]
concerns the region far from solitons
for which we will complete the square to obtain positivity.

\item The third term
\[
h_3 =\frac 1L \int \psi ( L \rho - |x| \Theta) (|\nabla\ep|^2+\eta^2)
\]
concerns the intermediate region,
since the function $\Phi$ is localized outside the soliton region, while for $|x|>Lt$, it holds $L \rho = |x|\Theta$ so that the integrant vanishes.
\item The fourth term contains an error term due to the cut-off around
each soliton
\[
h_4=
2 \int \psi \left(( {\bC} -x\Theta ) \cdot \nabla \ep\right) \eta
\]
since 
the function $\bC-x \Theta$ is localized on the soliton regions
(in some sense).
\item The fifth term is a similar error term
\[
h_5 = 
2 \int \Bigl(\Bigl( (1-\psi) \bC- \rho \sum_k \qk \bl_k \Bigr) \cdot \nabla\ep\Bigr)\eta
\]
since $1-\psi \approx \sum_k \qk$ and $\qk$ localizes on the $k$th soliton region, while
$\bC-\rho \bl_k\approx\bC- \bl_k t$ is small close to the $k$th soliton.
\item The last two terms are error terms due to the non linear interactions
\begin{align*}
h_6 &= - 2 \int \rho \Bigl(F (\bW +\ep )-F (\bW )-f (\bW )\ep -\frac 12 f'(\bW) \ep^2\Bigr) 
\\
h_7 &= \int \rho \Big(\sum_k f'(W_k) - f'(\bW)\Big) \ep^2 ,
\end{align*}
the term $h_6$ contains cubic terms and higher in $\ep$ while
the term $h_7$ is small since different solitons interact weakly
at large distances.
\end{itemize}
To check the above decomposition of $\cH$, we first decompose $\cH_1$.
Using
\begin{equation}\label{eq:dn}
\rho = \rho (1-\psi) + \rho\psi
=\rho(1-\psi) + \frac 1L\Phi |x| + \frac\psi L(L\rho - |x|\Theta),
\end{equation}
we obtain
\begin{align*}
\cH_1 & = 
\int \rho (1-\psi) \left(|\nabla\ep|^2+\eta^2\right) 
- \sum_k \int \rho f'(W_k) \ep^2\\
&\quad
+\int \frac 1L\Phi |x| \left(|\nabla\ep|^2+\eta^2\right) 
+\int \frac\psi L(L\rho - |x|\Theta)\left(|\nabla\ep|^2+\eta^2\right) \\
&\quad
-2 \int \rho \Big (F (\bW +\ep )-F (\bW )-f (\bW )\ep -\frac 12 \sum_kf'(W_k) \ep^2\Big) 
\end{align*}
so that $\cH_1$ contributes to $h_1$, $h_2$, $h_3$, $h_6$ and $h_7$.
Second,
\begin{align*}
\cH_2 & = 2 \int (1-\psi) (\bC\cdot \nabla\ep)\eta
+ 2 \int\psi (\bC\cdot\nabla\ep)\eta\\
&= 2 \sum_k \int \rho \qk (\bl_k \cdot \nabla\ep)\eta
+ 2 \int \Phi (x\cdot\nabla\ep)\eta\\
&\quad
+2 \int \psi \left( \bC - x\Theta ) \cdot \nabla \ep\right) \eta
+2 \int \Bigl(\Bigl( (1-\psi) \bC- \rho \sum_k \qk \bl_k \Bigr) \cdot \nabla\ep\Bigr)\eta,
\end{align*}
so that $\cH_2$ contributes to $h_1$, $h_2$, $h_4$ and $h_5$.
Lastly, the term $\cH_3$ only contributes to $h_2$.
\medskip

\emph{Estimate of $h_1$.}
We claim that for some constants $\mu_0>0$, $C>0$,
\begin{equation}\label{eq:h1}
h_1 
\geq \mu_0 \int \rho (1-\psi)(|\nabla \ep|^2 + \eta^2)
-\frac 1{\mu_0} \sum_k \left((z_k^-)^2+ (z_k^+)^2\right) 
 - C t^{-\alpha} \int \rho (|\nabla \ep|^2 + \eta^2).
\end{equation}
Define
\[
\tilde h_{1,k} =
\int \qk \left(|\nabla\ep|^2+\eta^2 
+ 2 (\bl_k\cdot \nabla \ep ) \eta \right)
- f'(W_k) \ep^2.
\]
By Lemma 2.2 (ii) in \cite{MaM1}, there exists $\mu_0\in(0,1)$ such that
\[
\tilde h_{1,k} \geq {\mu_0} \int (|\nabla \ep|^2 + \eta^2) \qk
-\frac 1{\mu_0}  \left((z_k^-)^2+ (z_k^+)^2\right).
\]
Now, we estimate the difference
\begin{align*}
h_1 - t \sum_k \tilde h_{1,k}
& = \int (|\nabla \ep|^2 + \eta^2) \Bigl(\rho(1-\psi) - t \sum_k \qk\Bigr)\\
&\quad+ \sum_k \int (\rho-t) \left(2\qk(\bl_k\cdot \nabla \ep ) \eta 
- f'(W_k) \ep^2 \right).
\end{align*}
Observe that
\[
\Bigl|\rho(1-\psi) - t \sum_k \qk\Bigr|\lesssim 
|\rho-t|(1-\psi)+ t \Bigl|1-\psi - \sum_k \qk\Bigr|
\lesssim \rho t^{-\alpha} 
\]
and 
\[
|\rho-t| \qk\lesssim \rho t^{-\alpha},\quad |\rho-t||f'(W_k)|\lesssim |\rho-t|q_k^4
\lesssim \rho q_k^{2} t^{-\alpha},
\]
so that using \eqref{eq:e1},
\[
\Bigl|h_1 - t \sum_k \tilde h_{1,k}\Bigr|
\lesssim t^{-\alpha} \int \rho (|\nabla \ep|^2 + \eta^2).
\]
The above estimates prove \eqref{eq:h1}.

\medskip

\emph{Estimate of $h_2$.}
We claim
\begin{equation}\label{eq:h2}
h_2 \geq \sigma \int \Phi |x| (|\nabla\ep|^2+\eta^2)
-C \alpha \int \rho|\nabla\ep|^2.
\end{equation}
We start with a simple yet key computation.
\begin{lemma}\label{le:sk}
It holds
\[
\int \Phi |x| \biggl(\frac{x}{|x|}\cdot\nabla\ep+\eta\biggr)^2
+8 \int \Phi\ep\eta
=\int \Phi |x| \biggl(\frac{x}{|x|}\cdot\nabla\ep+\eta+4 \frac{\ep}{|x|}\biggr)^2 
+ 4\int\biggl( \nabla \Phi\cdot \frac{x}{|x|}\biggr) \ep^2.
\]
\end{lemma}
\begin{proof}
We expand the square
\begin{align*}
\int \Phi |x| \biggl(\frac{x}{|x|}\cdot\nabla\ep+\eta+4 \frac{\ep}{|x|}\biggr)^2
&=\int \Phi |x| \biggl(\frac{x}{|x|}\cdot\nabla\ep+\eta\biggr)^2
+ 8 \int \Phi \ep \eta \\
&\quad + 8 \int \Phi \biggl(\frac{x}{|x|}\cdot\nabla\ep\biggr) \ep
+16\int \Phi\frac{\ep^2}{|x|}.
\end{align*}
Integrating by parts, we obtain
\begin{align*}
8 \int \Phi \biggl(\frac{x}{|x|}\cdot\nabla\ep\biggr) \ep
& = -4\int \DV \biggl( \Phi\frac{x}{|x|}\biggr) \ep^2\\
& = -4 \int \biggl( \nabla \Phi\cdot \frac{x}{|x|}\biggr)\ep^2 
- 4 \int \Phi \DV\biggl(\frac{x}{|x|}\biggr)\ep^2 .
\end{align*}
Since $\DV\bigl( {x}/{|x|}\bigr)= 4/{|x|}$, we obtain the result.
\end{proof}
Now, we decompose $h_2$ as follows
\begin{align*}
h_2 & = \left(\frac 1L-1\right)\int \Phi |x| (|\nabla\ep|^2+\eta^2)
+\int \Phi |x| \biggl(\frac{x}{|x|}\cdot\nabla\ep+\eta\biggr)^2
+4 \int \Phi \ep \eta \\
&\quad +\int \Phi\biggl(|x||\nabla\ep|^2-\frac{|x\cdot\nabla\ep|^2}{|x|}\biggr).
\end{align*}
The last term is non negative and we use Lemma \ref{le:sk} to obtain
\begin{align*}
h_2 & \geq \left(\frac 1L-1\right)\int \Phi |x| (|\nabla\ep|^2+\eta^2)
+\frac 12 \int \Phi |x| \biggl(\frac{x}{|x|}\cdot\nabla\ep+\eta\biggr)^2\\
&\quad +\frac 12 \int \Phi |x| \biggl(\frac{x}{|x|}\cdot\nabla\ep+\eta+4 \frac{\ep}{|x|}\biggr)^2 
+ 2\int\biggl( \nabla \Phi\cdot \frac{x}{|x|}\biggr) \ep^2.
\end{align*}
By \eqref{eq:pk} and \eqref{eq:Ph} and then \eqref{eq:e1},
\[
\biggl| \int\biggl( \nabla \Phi\cdot \frac{x}{|x|}\biggr) \ep^2\biggr|
\lesssim \alpha \int \Theta q\ep^2
\lesssim \alpha \int \rho |\nabla\ep|^2.
\]
Thus, \eqref{eq:h2} follows from $L=1-\sigma$.
\medskip

\emph{Estimate of $h_3$.}
We simply observe that by the definitions of $\rho$ and $\Theta$,
\[
h_3 =\frac 1L \int \psi(L\rho-|x|\Theta)(|\nabla\ep|^2+\eta^2)\geq 0.
\]

\emph{Conclusion for $h_1$, $h_2$ and $h_3$.}
The identity \eqref{eq:dn} allowed us to decompose $\cH_1$ into
different parts; we now use it again, to obtain a positivity result on the sum
$h_1+h_2+h_3$ (up to the instable directions).
Indeed, by \eqref{eq:h1}, \eqref{eq:h2} and the expression of $h_3$,
setting $\mu= \frac14\min(\mu_0,L\sigma) >0$, and
using \eqref{eq:dn}, we have proved
\[
h_1+h_2+h_3 \geq 4\mu \int \rho (|\nabla\ep|^2+\eta^2)-\frac 1{\mu} \sum_k \left((z_k^-)^2+ (z_k^+)^2\right)
-C\left(t^{-\alpha}+\alpha\right) \int \rho (|\nabla\ep|^2+\eta^2)
\]
and thus,
for $\alpha$ sufficiently small and for $t$ large (depending on $\alpha$ yet to be fixed),
\begin{equation}\label{eq:hh}
h_1+h_2+h_3 \geq 2\mu \int \rho (|\nabla\ep|^2+\eta^2)-\frac 1{\mu} \sum_k \left((z_k^-)^2+ (z_k^+)^2\right).
\end{equation}

\medskip

\emph{Estimate of $h_4$.}
We claim
\begin{equation}\label{eq:h4}
|h_4| \lesssim t^{-\alpha} \int \rho (|\nabla\ep|^2+|\eta|^2).
\end{equation}
Indeed, by $0\leq\psi\leq1$, \eqref{eq:dx} and then $|\BF_k|\lesssim t^{1-\alpha}$
from \eqref{eq:fk}, we have
\[
\psi|\bC - x\Theta| 
\lesssim  \sum_k \left(|x-\bl_k t - \by_k^\infty| \qk+|\BF_k|\right)
\lesssim \sum_k q_k^{-1+\alpha} + t^{1-\alpha}\lesssim t^{-\alpha} \rho,
\]
and \eqref{eq:h4} follows.

\medskip

\emph{Estimate of $h_5$.} We claim
\begin{equation}\label{eq:h5}
|h_5| \lesssim t^{-\alpha} \int \rho (|\nabla\ep|^2+|\eta|^2).
\end{equation}
We decompose
\[
h_5 = 
2 \int \Bigl( \Bigl(1-\psi-\sum_k q_k^\alpha\Bigr)\Bigr) (\bC\cdot \nabla\ep)\eta
+2\sum_k \int \qk ((\bC - \rho \bl_k )\cdot \nabla\ep)\eta.
\]
For the first term on the right-hand side, since 
$|1-\psi-\sum_k q_k^\alpha|\lesssim t^{-\alpha}$ 
and $|\bC|\lesssim \rho$, we have
\[
\Bigl|\int \Bigl( \Bigl(1-\psi-\sum_k q_k^\alpha\Bigr)\Bigr) (\bC\cdot \nabla\ep)\eta\Bigr|\lesssim t^{-\alpha} \int \rho (|\nabla\ep|^2+\eta^2).
\]
For the second term, we fix $k$ and we distinguish the two cases $x\in B_k$
and $x\not\in B_k$.
If $x\in B_k$ then using \eqref{eq:cl},
$\qk|\bC-\rho \bl_k|=\qk |\bC-t \bl_k|\lesssim q_k^{-1+\alpha}\lesssim t^{1-\alpha}= t^{-\alpha}\rho$.
Thus,
\[
\Bigl|\int_{B_k}\int \qk((\bC - \rho\bl_k )\cdot \nabla\ep)\eta\Bigr|
\lesssim t^{-\alpha} \int \rho (|\nabla\ep|^2+|\eta|^2).
\]
If $x\not\in B_k$, then $\qk\lesssim t^{-\alpha}$ and
$|\bC-\bl_k \rho|\lesssim \rho$ and so
\[
\Bigl|\int_{\RR^5\setminus B_k}\int \qk((\bC-\rho \bl_k )\cdot \nabla\ep)\eta\Bigr|
\lesssim t^{- \alpha} \int \rho (|\nabla\ep|^2+|\eta|^2).
\]
The estimate \eqref{eq:h5} is thus proved.

\medskip

\emph{Estimate of $h_6$.} We claim
\begin{equation}\label{eq:h6}
|h_6| \lesssim t^{-\frac12} \int \rho |\nabla\ep|^2 .
\end{equation}
Note that
\[
\Bigl|F (\bW +\ep )-F (\bW )-f (\bW )\ep -\frac 12 f'(\bW) \ep^2\Bigr|
\lesssim |\bW|^{\frac 13} |\ep|^3 + |\ep|^{\frac{10}3}.
\]
Recall from \eqref{eq:bW} that 
$|\bW|\lesssim q^{3}+t^{-1}q^{2}\lesssim q^2$.
Thus, by the Hölder inequality,
\[
\int\rho|\bW|^\frac13|\ep|^3
\lesssim \left(\int |\rho^\frac13\ep|^\frac{10}3\right)^\frac{9}{10}
\left(\int |\bW|^\frac{10}3\right)^{\frac1{10}} 
\lesssim \left(\int |\nabla (\rho^\frac13\ep)|^2\right)^{\frac32}.
\]
Then, by the Sobolev inequality and
\eqref{eq:e1}, we obtain
\begin{align*}
\left(\int |\nabla (\rho^\frac13\ep)|^2\right)^{\frac32}
&\lesssim \left(\int \rho^{\frac23}|\nabla\ep|^2\right)^{\frac32}
+\left(\int\rho^{\frac23}|x|^{-2}\ep^2\right)^{\frac32}\\
&\lesssim t^{-\frac12} \left(\int \rho |\nabla\ep|^2\right)^{\frac32}
\lesssim t^{-\frac12} \int \rho |\nabla\ep|^2 ,
\end{align*}
where by \eqref{eq:BS}, we have estimated $\int \rho |\nabla\ep|^2\lesssim1$.
Similarly,
\[
\int \rho|\ep|^\frac{10}3
\lesssim\left(\int |\nabla (\rho^\frac3{10}\ep)|^2\right)^{\frac53}
\lesssim t^{-\frac23} \left(\int \rho |\nabla\ep|^2\right)^{\frac53}
\lesssim t^{-\frac23} \int \rho |\nabla\ep|^2 .
\]

\emph{Estimate of $h_7$.} We claim
\begin{equation}\label{eq:h7}
|h_7| \lesssim t^{-1}\|\rho\nabla \ep\|_{L^2}^2.
\end{equation}
On $\Omega=\RR^5\setminus\cup_mB_m$, using the expression of $W_k$ and
\eqref{eq:Av}, one has
\[
|W_k|+|v_k|\lesssim t^{-1}q^2
\]
and so
\[
\Bigl|\sum_k f'(W_k)-f'(\bW)\Bigr|
\lesssim \sum_k |W_k|^\frac43+|\bW|^\frac43\lesssim t^{-1}q^2.
\]
On $B_k$, for $m\neq k$, one has $|W_m|+|v_k|+|v_m|\lesssim t^{-1}q^2$
and thus
\begin{align*}
\Bigl|\sum_k f'(W_k)-f'(\bW)\Bigr|
&\lesssim \sum_{m\neq k} |W_m|^\frac43
+|W_k|^\frac43\Bigl|1-\Bigl|1+\sum_{m\neq k} W_m W_k^{-1} + \sum_m v_m W_k^{-1}\Bigr|^\frac43\Bigr|\\
&\lesssim \sum_{m\neq k} |W_m|^\frac43
+|W_k|^\frac13\Bigl(\sum_{m\neq k}|W_m|+\sum_m |v_m|\Bigr)\lesssim t^{-1}q^2.
\end{align*}
Therefore, using \eqref{eq:e1},
\[
|h_7|\lesssim t^{-1} \int \rho q^2 \ep^2 \lesssim t^{-1} \int \rho|\nabla\ep|^2.
\]

In conclusion, using \eqref{eq:hh}, \eqref{eq:h4}, \eqref{eq:h5}, \eqref{eq:h6}
and \eqref{eq:h7}, we obtain
\begin{align*}
\cH 
&\geq 2\mu \int \rho (|\nabla\ep|^2+\eta^2)-\frac 1{\mu} \sum_k \left((z_k^-)^2+ (z_k^+)^2\right)
- C t^{-\alpha} \int \rho (|\nabla\ep|^2+\eta^2)\\
&\geq \mu \int \rho (|\nabla\ep|^2+\eta^2)-\frac 1{\mu} \sum_k \left((z_k^-)^2+ (z_k^+)^2\right)
\end{align*}
for $t$ sufficiently large.
\end{proof}

\subsection{Variation of the energy}
\begin{lemma}\label{le:en}
For $\alpha>0$ small enough, 
for all $t\in I$, it holds
\[
\cH' + \frac 2 {L^2}\frac \cH t \gtrsim_\delta 
- t^{-6+3\delta} .
\]
\end{lemma}
\begin{proof}
We compute the time derivatives of $\cH_1$, $\cH_2$ and $\cH_3$ using the systems
\begin{equation}\label{eq:sW}
\left\{\begin{aligned}
\partial_t\bW & = \bX - \bM_{\bW} - \bR_{\bW}\\
\partial_t	\bX & 
	= \Delta \bW +|\bW|^{\frac 43} \bW - \bM_{\bX}- \bR_{\bX}
\end{aligned}\right.
\end{equation} 
and
\begin{equation}\label{eq:EE}
\left\{\begin{aligned}
\partial_t \ep & = \eta + \bM_{\bW}+\bR_{\bW}\\
\partial_t \eta & 
= \Delta \ep +f(\bW + \ep)- f(\bW) + \bM_{\bX}+ \bR_{\bX} .
\end{aligned}\right.
\end{equation} 

First,
\begin{align*}
\frac d{dt} \cH_1 & = \int (\partial_t\rho) \left(|\nabla\ep|^2+\eta^2- 2 (F (\bW +\ep )-F (\bW )-f (\bW )\ep ) \right) \\
&\quad +2 \int \rho (\nabla\partial_t\ep\cdot\nabla\ep) +(\partial_t \eta)\eta- (\partial_t \ep) \left(f (\bW +\ep )-f (\bW )\right) \\
&\quad -2 \int \rho(\partial_t \bW) \left(f(\bW +\ep )-f(\bW )-f'(\bW )\ep \right).
\end{align*}
Thus, using the expression of $\rho$ and integrating by parts, 
\begin{align*}
\frac{d}{dt} \cH_1 & = 
\int_{|x|< Lt} \left(|\nabla\ep|^2+\eta^2- 2 (F (\bW +\ep )-F (\bW )-f (\bW )\ep ) \right)\\
&\quad +\frac{\alpha}t \int_{|x|>Lt} \rho\left(|\nabla\ep|^2+\eta^2- 2 (F (\bW +\ep )-F (\bW )-f (\bW )\ep ) \right)\\
&\quad -2 (1-\alpha) \int_{|x|> Lt} \frac{\rho}{|x|^2} (x\cdot \nabla \ep )(\partial_t \ep) \\ 
&\quad +2 \int \rho (\partial_t \ep) \left( -\Delta \ep - \left( f(\bW + \ep) - f(\bW) \right)\right) +2 \int \rho(\partial_t \eta) \eta\\
& \quad -2 \int \rho (\partial_t \bW) \left( f(\bW + \ep) - f(\bW)-f'(\bW) \ep \right). \end{align*}
Using~\eqref{eq:sW} and~\eqref{eq:EE}
\begin{align*}
\frac{d}{dt} \cH_1 & =
\int_{|x|< Lt} \left(|\nabla\ep|^2+\eta^2- 2 (F (\bW +\ep )-F (\bW )-f (\bW )\ep ) \right)\\
&\quad +\frac{\alpha}t \int_{|x|>Lt} \rho\left(|\nabla\ep|^2+\eta^2- 2 (F (\bW +\ep )-F (\bW )-f (\bW )\ep ) \right)\\
&\quad - 2(1-\alpha) \int_{|x|> Lt} \frac{\rho}{|x|^2} (x\cdot \nabla \ep )\eta
- 2(1-\alpha) \int_{|x|> Lt} \frac{\rho}{|x|^2} (x\cdot \nabla \ep ) (\bM_\bW +\bR_\bW)\\
&\quad + 2 \int \rho \left(-\Delta \ep -f'(\bW)\ep\right)(\bM_\bW +{\bR}_{\bW})
+ 2 \int \rho\eta(\bM_\bX+\bR_\bX) \\& \quad -2 \int \rho\bX \left( f(\bW + \ep) - f(\bW)-f'(\bW) \ep\right).
\end{align*}

Second,
\begin{align*}
\frac d{dt} \cH_2 & = 2 \int (\bC \cdot \nabla \partial_t \ep) \eta
+ 2 \int (\bC \cdot \nabla \ep) \partial_t \eta 
+ 2 \int ( \partial_t \bC \cdot \nabla \ep ) \eta\\
& = 2 \int (\bC \cdot \nabla \eta) \eta 
+ 2 \int (\bC \cdot \nabla \ep) \left( \Delta \ep + \left( f(\bW + \ep) - f(\bW) \right) \right)+ 2 \int ( \partial_t \bC \cdot \nabla \ep ) \eta
 \\& \quad+2 \int (\bC \cdot \nabla \bM_{\bW}) \eta+2 \int (\bC \cdot \nabla {\bR}_{\bW}) \eta
 + 2 \int (\bC \cdot \nabla \ep) \bM_{\bX}
 + 2 \int (\bC \cdot \nabla\ep) {\bR}_{\bX}.
\end{align*}
Let us denote $J_{\bC}=(\partial_i\bC_j)_{j,i}$ the Jacobian matrix of $\bC$.
Then, by \eqref{eq:GG}, one has
\begin{equation}\label{eq:G2}
J_{\bC}=\Phi \,\ID + G \quad \mbox{where
$G=(G_{j,i})_{j,i}$ satisfies $|G|\lesssim \alpha\Theta$.}
\end{equation}
From the notation \eqref{eq:G2}, $\DV\bC = 5\Phi + \tr G$, one gets
\[
2 \int (\bC \cdot \nabla \eta) \eta
=-\int \eta^2 \DV \bC = - 5 \int \eta^2 \Phi -\int \eta^2 \tr G.
\]
Moreover, by integration by parts and using \eqref{eq:G2},
\begin{align*}
2 \int (\bC \cdot \nabla \ep) \Delta \ep 
& = 
\int \DV \bC \, |\nabla\ep|^2-2\int  (\nabla\ep)^\TR  J_{\bC} \nabla\ep \\
& = 3 \int |\nabla \ep|^2 \Phi 
+ \int  |\nabla\ep|^2\tr G -2\int (\nabla\ep)^\TR G\, \nabla\ep.
\end{align*}
Lastly, using again $\DV\bC = 5\Phi + \tr G$,
\begin{align*}
2 \int (\bC \cdot \nabla \ep) \left( f(\bW + \ep) - f(\bW) \right) 
& = 
- 2 \int (\bC \cdot \nabla \bW) \left( f(\bW + \ep) - f(\bW) -f'(\bW)\ep\right) \\
&\quad + 2 \int ( \bC \cdot \nabla ) \left( F(\bW + \ep) - F(\bW)-f(\bW) \ep \right)\\
& = 
- 2 \int (\bC \cdot \nabla \bW) \left( f(\bW + \ep) - f(\bW) -f'(\bW)\ep\right) \\
&\quad -10 \int \left( F(\bW + \ep) - F(\bW)-f(\bW) \ep \right) \Phi\\
&\quad -2 \int \left( F(\bW + \ep) - F(\bW)-f(\bW) \ep \right) \tr G.
\end{align*}
Therefore,
\begin{align*}
\frac d{dt} \cH_2 
& = - 5 \int \eta^2 \Phi+3 \int |\nabla \ep|^2 \Phi 
-10 \int \left( F(\bW + \ep) - F(\bW)-f(\bW) \ep \right) \Phi\\
&\quad - 2 \int (\bC \cdot \nabla \bW) \left( f(\bW + \ep) - f(\bW) -f'(\bW)\ep\right) + 2 \int ( \partial_t \bC \cdot \nabla \ep ) \eta\\
& \quad
+\int |\nabla \ep|^2 \tr G - 2 \int (\nabla\ep)^\TR G \,\nabla \ep
-\int \eta^2 \tr G
\\& \quad+2 \int (\bC \cdot \nabla \bM_{\bW}) \eta+2 \int (\bC \cdot \nabla {\bR}_{\bW}) \eta
+ 2 \int (\bC \cdot \nabla \ep) \bM_{\bX}
 + 2 \int (\bC \cdot \nabla\ep) {\bR}_{\bX}\\
&\quad -2 \int \left( F(\bW + \ep) - F(\bW)-f(\bW) \ep \right) \tr G.
\end{align*}

Third,
\begin{align*}
 \frac d{dt} \cH_3
& = 4 \int \Phi (\partial_t\ep) \eta+4 \int \Phi \ep (\partial_t\eta)
+4 \int (\partial_t\Phi) \ep \eta\\
& = 4 \int \eta^2 \Phi+4 \int \Phi \ep (\Delta\ep)
+ 4 \int (f(\bW+\ep)-f(\bW))\ep \Phi \\
& \quad + 4 \int \bM_{\bW} \eta \Phi + 4 \int \bM_{\bX} \ep \Phi
+ 4 \int {\bR}_{\bW} \eta \Phi + 4 \int {\bR}_{\bX} \ep \Phi+4 \int (\partial_t\Phi) \ep \eta\\
& = 4 \int \eta^2 \Phi- 4 \int |\nabla \ep|^2\Phi
-4 \int \ep (\nabla \ep\cdot \nabla\Phi)
+ 4 \int (f(\bW+\ep)-f(\bW))\ep \Phi \\
& \quad + 4 \int \bM_{\bW} \eta \Phi + 4 \int \bM_{\bX} \ep \Phi
+ 4 \int {\bR}_{\bW} \eta \Phi + 4 \int {\bR}_{\bX} \ep \Phi
+4 \int (\partial_t\Phi) \ep \eta.
\end{align*}
The above computations imply that
\[
\frac d{dt} \cH 
 = \frac d{dt} \cH_1+ \frac d{dt} \cH_2 + \frac d{dt} \cH_3
 =g_1+g_2+g_3+g_4+g_5
\]
where we gather the various terms as described below.
\begin{itemize}
\item The first term contains the main terms
\begin{align*}
g_1 & = - \int_{|x|> Lt} \Phi(|\nabla \ep|^2+\eta^2) 
- 2 (1-\alpha) \int_{|x|> Lt} \frac{\rho}{|x|^2} (x\cdot \nabla \ep )\eta \\
&\quad +
\int_{|x|< Lt} \bigg( \left(|\nabla\ep|^2  +\eta^2 \right) \sum_k \qk
- 2 (F (\bW +\ep )-F (\bW )-f (\bW )\ep ) \bigg)\\
&\quad + 2 \int\sum_k \qk (\bl_k \cdot \nabla \ep ) \eta
\end{align*}
which we will estimate using $\cH$.

\item The second term, defined by
\[
g_2= -2 \int (\rho\bX+(\bC \cdot \nabla \bW)) \left( f(\bW + \ep) - f(\bW)-f'(\bW) \ep\right)
\]
will be estimated using the fact that 
the function $\rho\bX+(\bC \cdot \nabla \bW)$ is small since
$\bC\sim \bl_k\rho $ close to a soliton $W_k$ and 
$X_k=-\bl_k \nabla W_k$.
\item Terms containing the error terms $\bM_\bW$ and $\bM_\bX$ are
gathered in $g_3$
\begin{align*}
g_3 & = - 2(1-\alpha) \int_{|x|> Lt} \frac{\rho}{|x|^2} (x\cdot \nabla \ep ) \bM_\bW 
+ 2 \int \rho \left(-\Delta \ep -f'(\bW)\ep\right)\bM_\bW \\
&\quad +2 \int (\bC \cdot \nabla \ep) \bM_{\bX}
+ 4 \int \ep \bM_{\bX} \Phi
 + 2 \int \eta \left( \rho \bM_\bX + \bC\cdot \nabla \bM_\bW
+ 2\Phi \bM_\bW \right).
\end{align*}
\item Terms containing the error terms $\bR_\bW$ and $\bR_\bX$ are
gathered in $g_3$
\begin{align*}
g_4 & = -2 (1-\alpha) \int_{|x|> Lt} \frac{\rho}{|x|^2} (x\cdot \nabla \ep )\bR_\bW 
+ 2 \int \rho \left(-\Delta \ep -f'(\bW)\ep\right)\bR_\bW \\
&\quad +2 \int (\bC \cdot \nabla \ep) \bR_{\bX}
+ 4 \int \ep \bR _{\bX}  \Phi+ 2 \int \eta \left(\rho \bR_\bX + \bC\cdot \nabla \bR_\bW
+ 2 \Phi \bR_\bW \right) .
\end{align*}
\item The last term gathers all other error terms
\begingroup
\allowdisplaybreaks
\begin{align*}
g_5 & = 
 -4 \int \ep (\nabla \ep\cdot \nabla\Phi)
+2 \int \Bigl(\Bigl(\partial_t \bC - \sum_k \bl_k \qk\Bigr)\cdot \nabla \ep\Bigr) \eta+4 \int (\partial_t\Phi) \ep \eta\\
& \quad
+\int_{|x|<Lt} \Bigl(1-\Phi-\sum_k \qk\Bigr)\left(|\nabla\ep|^2+\eta^2\right)\\
&\quad 
+\int |\nabla \ep|^2 \tr G - 2 \int (\nabla\ep)^\TR G \,\nabla \ep
-\int \eta^2 \tr G\\
&\quad 
-10\int (F(\bW+\ep)-F(\bW)-f(\bW) \ep) \Phi
+ 4 \int (f(\bW+\ep)-f(\bW)) \ep \Phi\\
&\quad -2 \int \left( F(\bW + \ep) - F(\bW)-f(\bW) \ep \right) \tr G\\
&\quad
+\frac{\alpha}t \int_{|x|>Lt} \rho\left(|\nabla\ep|^2+\eta^2- 2 (F (\bW +\ep )-F (\bW )-f (\bW )\ep ) \right) .
\end{align*}
\endgroup
\end{itemize}

\medskip

\emph{Estimate of $ g_1$.}
We claim that for some constants  $ \mu_1>0$, $C>0$,
\begin{equation}\label{eq:g1}
g_1+\frac{2}{L^2 t } \cH \geq \frac{\mu_1} t \int \rho (|\nabla\ep|^2+\eta^2)
- C\sum_k \left((z_k^-)^2+ (z_k^+)^2\right)
\end{equation}
We rearrange $g_1$ as follows
\begin{align*}
g_1 & = - \int_{|x|> Lt} \Phi(|\nabla \ep|^2+\eta^2) 
- 2(1-\alpha) \int_{|x|> Lt} \frac{\rho}{|x|^2} (x\cdot \nabla \ep )\eta\\
&\quad +
\sum_k \int_{|x|< Lt} \left(\qk \left(|\nabla\ep|^2 +\eta^2 
+ 2 (\bl_k \cdot \nabla \ep ) \eta\right) - f'(W_k )\ep^2\right)\\
&\quad - 2\int_{|x|< Lt} (F (\bW +\ep )-F (\bW )-f (\bW )\ep )
+\sum_k \int f'(W_k) \ep^2.
\end{align*}
We define
\[
\tilde g_1 = -\int_{|x|> Lt} \Phi(|\nabla \ep|^2+\eta^2) 
-\frac 2L(1-\alpha)\int_{|x|> Lt} \Phi \biggl(\frac{x}{|x|}\cdot \nabla \ep\biggr)\eta 
+ \frac 2{L^2 t} {h_2} ,
\]
where $h_2$ was defined in the proof of Lemma \ref{le:co}.
Recall that
\begin{align*}
h_2 & = \frac\sigma L\int \Phi |x| (|\nabla\ep|^2+\eta^2) 
+ \int \Phi |x| (|\nabla\ep|^2+\eta^2)
+2 \int \Phi(x\cdot \nabla \ep) \eta +4 \int \Phi \ep \eta.
\end{align*}
Thus, using $(1-\alpha)/L\geq 1$ for $0<\alpha\leq \sigma$,
\begin{align*}
\tilde g_1 & \geq
-\frac {1-\alpha}L\int_{|x|> Lt} \Phi(|\nabla \ep|^2+\eta^2) 
-\frac 2L(1-\alpha)\int_{|x|> Lt} \Phi \biggl(\frac{x}{|x|}\cdot \nabla \ep\biggr)\eta \\
&\quad + \frac{2}{L^2t} \biggl(\int \Phi |x| (|\nabla\ep|^2+\eta^2)
+2 \int \Phi(x\cdot \nabla \ep) \eta +4 \int \Phi \ep \eta\biggr)\\
&\quad + \frac{2\sigma}{L^3t}\int \Phi |x| (|\nabla\ep|^2+\eta^2).
\end{align*}
Forming squares, we obtain
\begin{align*}
\tilde g_1 & \geq -\frac {1-\alpha}L \int_{|x|> Lt} \Phi \biggl(\frac{x}{|x|}\cdot \nabla \ep+\eta\biggr)^2 
-\frac {1-\alpha}L\int_{|x|> Lt} \Phi\biggl( |\nabla\ep|^2-\frac{|x\cdot\nabla\ep|^2}{|x|^2}\biggr)\\
&\quad + \frac 2{L^2t} \int \Phi|x| \biggl(\frac{x}{|x|}\cdot \nabla \ep+\eta\biggr)^2 
+\frac 2{L^2t}\int \Phi |x|\biggl( |\nabla\ep|^2-\frac{|x\cdot\nabla\ep|^2}{|x|^2}\biggr) +\frac{8}{L^2t}\int \Phi \ep \eta\\
&\quad + \frac{2\sigma}{L^3t}\int \Phi |x| (|\nabla\ep|^2+\eta^2).
\end{align*}
Thus,
\[
\tilde g_1 \geq 
\frac{2\sigma}{L^3t}\int \Phi |x| (|\nabla\ep|^2+\eta^2)+\frac 1{L^2t} \int \Phi|x| \biggl(\frac{x}{|x|}\cdot \nabla \ep+\eta\biggr)^2 +\frac{8}{L^2t}\int \Phi \ep \eta.
\]
Using Lemma \ref{le:sk}, we obtain
\begin{align*}
\tilde g_1 
&\geq 
\frac{2\sigma}{L^3t}\int \Phi |x| (|\nabla\ep|^2+\eta^2) +
\frac 1{L^2t}
\int \Phi |x| \biggl(\frac{x}{|x|}\cdot\nabla\ep+\eta+4 \frac{\ep}{|x|}\biggr)^2 
\\
&\quad+\frac 4{L^2t}\int\biggl( \nabla \Phi\cdot \frac{x}{|x|}\biggr) \ep^2
\end{align*}
By \eqref{eq:pk} and then \eqref{eq:e1},
\[
\biggl|\frac 1t\int\biggl( \nabla \Phi\cdot \frac{x}{|x|}\biggr) \ep^2\biggr|
\lesssim \frac{\alpha}{t} \int \Theta q\ep^2
\lesssim \frac{\alpha}{t} \int \rho |\nabla\ep|^2.
\]
We have proved 
\begin{equation}\label{eq:gt}
\tilde g_1 \gtrsim \frac{2\sigma}{L^3t}\int \Phi |x| (|\nabla\ep|^2+\eta^2)- \frac\alpha{t}\int \rho |\nabla\ep|^2.
\end{equation}
Now, we decompose, using the notation of the proof of Lemma \ref{le:co},
\begin{align*}
g_1+\frac{2}{L^2t}\cH  & = \tilde g_1- 2(1-\alpha) \int_{|x|> Lt}\frac1{|x|} \Bigl(\frac{\rho}{|x|}- \frac \Phi L\Bigr) (x\cdot \nabla \ep )\eta\\
&\quad 
+\sum_k \int_{|x|< Lt} \left(\qk \left(|\nabla\ep|^2 +\eta^2 
+ 2 (\bl_k \cdot \nabla \ep ) \eta\right) - f'(W_k )\ep^2\right)\\
&\quad - 2\int_{|x|< Lt} (F (\bW +\ep )-F (\bW )-f (\bW )\ep )
+\sum_k \int f'(W_k) \ep^2\\
&\quad +\frac{2}{L^2t} (h_1+h_3+h_4+h_5+h_6+h_7).
\end{align*}
For the second term on the right-hand side, we use that for $|x|>Lt$,
$L\rho =|x| \Theta=|x|\Phi/\psi $ and 
$|1-\psi|\lesssim t^{-\alpha}$ so that
\[
\biggl|\int_{|x|> Lt}\frac1{|x|} \Bigl(\frac{\rho}{|x|}- \frac \Phi L\Bigr) (x\cdot \nabla \ep )\eta\biggr|
\lesssim t^{-1-\alpha} \int \rho \left(|\nabla\ep|^2+\eta^2\right).
\]
For the third term, we observe that
\[
\int_{|x|< Lt} \left(\qk \left(|\nabla\ep|^2 +\eta^2 
+ 2 (\bl_k \cdot \nabla \ep ) \eta\right) - f'(W_k )\ep^2\right)
\geq \tilde h_{1,k} - C t^{-1-\alpha} \int \rho \left(|\nabla\ep|^2+\eta^2\right)
\]
where
$\tilde h_{1,k}$, defined in the proof of Lemma \ref{le:co}, satisfies
\[
\tilde h_{1,k} \geq {\mu_0} \int (|\nabla \ep|^2 + \eta^2) \qk
-\frac 1{\mu_0} \sum_k \left((z_k^-)^2+ (z_k^+)^2\right).
\]
For the next terms, we use the estimates \eqref{eq:h1}--\eqref{eq:h7} of the proof of Lemma \ref{le:co}, where \eqref{eq:gt} now replaces \eqref{eq:h2}.
We deduce that there exists a constant $ \mu_1>0$ such that
\[
g_1+\frac{2}{L^2t}\cH\geq
\frac{2\mu_1} t \int \rho (|\nabla\ep|^2+\eta^2)
- C\sum_k \left((z_k^-)^2+ (z_k^+)^2\right)
- \frac{1}{t}\left(\alpha+t^{-\alpha}\right) \int \rho \left(|\nabla\ep|^2+\eta^2\right).
\]
This implies \eqref{eq:g1} for $\alpha>0$ sufficiently small and $t$ sufficiently large (depending on $\alpha$).

\medskip

\emph{Estimate of $g_2$.}
We claim
\begin{equation}\label{eq:g2}
|g_2|\lesssim \frac {\alpha^\frac12} t \int \rho |\nabla \ep|^2.
\end{equation}
On the one hand, by the definition of $\bW$, $\bX$, and \eqref{eq:Av}, we have
\[|\bX+\sum_k \bl_k\cdot \nabla W_k|
+|\nabla\bW - \sum_k \nabla W_k|
\lesssim \sum_k |z_k| + \sum_k |\nabla v_k|
\lesssim_\delta t^{-2} q^{-3+\delta}\]
and by \eqref{eq:cl},
$q_k^3 |\bC -\bl_k t | \lesssim \alpha^\frac 12 q_k$.
Thus, for $t$ sufficiently large,
\begin{align*}
|\rho\bX+(\bC \cdot \nabla \bW)|&\lesssim |\rho -t| |\bX| + t |\bX+\sum_k \bl_k\cdot \nabla W_k|
\\&\quad + | \sum_k (\bC-\bl_k t)\cdot \nabla W_k|
+|\bC (\sum_k \nabla W_k-\nabla \bW)|\\
& \lesssim_\delta |\rho -t| q^{3-\delta} 
+t^{-1} q^{3-\delta} + \alpha^\frac12 q^2
+\rho t^{-2}q^{3-\delta} \lesssim \alpha^\frac12 q^\frac32.
\end{align*}
On the other hand,
\[
\left| f(\bW + \ep) - f(\bW)-f'(\bW) \ep\right|
\lesssim |\bW| |\ep|^2 +|\ep|^\frac73
\lesssim q |\ep|^2+|\ep|^\frac73.
\]
Thus, using also \eqref{eq:e1}, the Hardy inequality, the Sobolev
inequality and \eqref{eq:BS},
\begin{align*}
|g_2| &\lesssim \alpha^\frac12 \int q^2 |\ep|^2 
+ \alpha^\frac12 \int q^\frac32 |\ep|^\frac 73\\
&\lesssim \alpha^\frac12 \int q^2\ep^2 + 
\alpha^\frac12 \left(\int q^2\ep^2\right)^\frac34\left(\int |\ep|^\frac{10}{3}\right)^\frac14 
 \lesssim \alpha^\frac12\int |\nabla \ep|^2 \lesssim \frac {\alpha^\frac12} t \int \rho |\nabla \ep|^2.
\end{align*}

\medskip

\emph{Estimate of $g_3$.}
We claim
\begin{equation}\label{eq:g3}
|g_3|
\lesssim \frac{\alpha^\frac1{10}}t \int \rho(|\nabla \ep|^2 + \eta^2) + t^{-6}.
\end{equation}
Integrating by parts the term containing $\Delta\ep$
and using $\nabla\rho=(1-\alpha) x|x|^{-2}\rho$ for $|x|>Lt$, there is a cancellation which leads to the
following expression for $g_3$
\begin{align*}
g_3& = 2 \int \rho\nabla\ep\cdot\nabla \bM_\bW
-2 \int \ep f'(\bW) \rho\bM_\bW +2 \int (\bC \cdot \nabla \ep) \bM_{\bX}
+ 4 \int \bM _{\bX} \ep \Phi\\
&\quad + 2 \int \eta \left(\rho \bM_\bX + \bC\cdot \nabla \bM_\bW
+ 2 \Phi \bM_\bW \right).
\end{align*}
Integrating by parts again, we obtain
\begin{align*}
g_3 & = - 2(1-\alpha) \int_{|x|>Lt} \ep\rho | x|^{-2}(x\cdot \nabla\bM_\bW)
-2 \int \ep \left( (\DV \bC)  -2\Phi \right)\bM_\bX 
+ 4 \int \eta \Phi \bM_\bW \\
&\quad 
+ 2 \int \eta \left( \rho \bM_\bX + \bC\cdot \nabla \bM_\bW\right)
-2 \int \ep \left(\rho \Delta \bM_\bW+f'(\bW) \rho\bM_\bW -\bC\cdot\nabla \bM_\bX \right)\\
&\quad 
=g_{3,1}+g_{3,2}+g_{3,3}+g_{3,4}+g_{3,5}.
\end{align*}
First, by \eqref{eq:Mm}, then the Cauchy inequality, the Hardy inequality
and then \eqref{eq:en},
\begin{align*}
|g_{3,1}|\lesssim t^{-3+2\delta} \int_{|x|>Lt} q^4 |\ep|
\lesssim t^{-3+2\delta}t^{-\frac12}\left(\int |x|^{-2}|\ep|^2\right)^\frac12
\lesssim t^{-\frac{13}2+4\delta}\lesssim t^{-6}.
\end{align*}
For the terms $g_{3,2}$ and $g_{3,3}$, we need to use
the (weak) localizing effect of $\Phi$ outside the solitons.
Recall that by \eqref{eq:G2}, we have
\[
(\DV \bC) -2\Phi= 3 \Phi + \tr G \quad\mbox{where}\quad |G|\lesssim \alpha.
\]
Moreover, by the simple inequality
$1-(1+a^2)^{-\alpha/2} \lesssim \alpha^\frac1{10} a^\frac15$, we have
\[
(1-\qk)\lesssim \alpha^\frac 1{10} q_k^{-\frac 15}.
\]
Using \eqref{eq:pC} and $0\leq \Phi\leq\psi$, this implies 
\[
\Phi |\bM_\bX|\lesssim \alpha^\frac 1{10} (t^{-\frac12}\cN+t^{-4}) q^{4-\frac15},
\quad
\Phi |\bM_\bW|\lesssim \alpha^\frac 1{10} (t^{-\frac12}\cN+t^{-4}) q^{3-\frac15}.
\]
Thus, using also \eqref{eq:e1}
and \eqref{eq:en} we deduce that
\begin{align*}
|g_{3,2}|&\lesssim 
\alpha^\frac 1{10} (t^{-\frac12}\cN+t^{-4}) \int |\ep|q^{4-\frac15}\\
&\lesssim \alpha^\frac 1{10} (t^{-\frac12}\cN+t^{-4})\left(\int q^2\ep^2\right)^\frac12
\left(\int q^{\frac{28}5}\right)^\frac 12 \\
&\lesssim \alpha^\frac 1{10}  t^{-1}\cN^2 +t^{-6}.
\end{align*}
Similarly,
\[
|g_{3,3}|\lesssim 
\alpha^\frac 1{10} (t^{-\frac12}\cN+t^{-4}) \int |\eta|q^{3-\frac15}
\lesssim \alpha^\frac 1{10} (t^{-\frac12}\cN+t^{-4})\|\eta\|_{L^2}
\lesssim \alpha^\frac 1{10} t^{-1} \cN^2 +t^{-6}.
\]
To estimate $g_{3,4}$ and $g_{3,5}$, we need to exploit cancellations. 
Set 
\[
M_k = \Biggl(\frac {\dot \lambda_k}{\lambda_k} - \frac{a_k}{\lambda_k^{\frac 12}t^2}\Biggr)\Lambda W_k + \dot{\by}_k\cdot \nabla W_k \]
so that
$\bM_{\bW} = \sum_k M_k$ and $\bM_{\bX} = - \sum_k \bl_k \cdot \nabla M_k$.
In particular,
\[
|t\bM_\bX+\bC\cdot\nabla\bM_\bW|
=\biggl|\sum_k (\bC- \bl_kt)\cdot\nabla M_k\biggr|.
\]
From \eqref{eq:cl}, one has
$|\bC-\bl_k t|\lesssim \alpha^\frac1{10} q_k^{-\frac 65}$ on $B_k$, and so
\begin{equation}\label{eq:cm}
\int (\bC-\bl_k t)^2 q_k^8
\lesssim \int_{B_k} + \int_{\RR^5\setminus B_k}
\lesssim \alpha^\frac15 \int q_k^\frac{28}5 + \int_{\RR^5\setminus B_k} q_k^6
\lesssim \alpha^\frac15 + t^{-1}\lesssim \alpha^\frac1{5}
\end{equation}
for $t$ sufficiently large.
Thus, by \eqref{eq:pC} and then \eqref{eq:en},
\begin{align*}
\biggl|\int \eta (t \bM_\bX+\bC\cdot\nabla\bM_\bW)\biggr|
&\lesssim (t^{-\frac12}\cN+t^{-4})
\sum_k \int |\eta| |\bC-\bl_kt| q_k^4\\
&\lesssim (t^{-\frac12}\cN+t^{-4}) \|\eta\|_{L^2} \|(\bC-\bl_k t) q_k^4\|_{L^2}\\
& \lesssim \alpha^\frac1{10} t^{-1}\cN^2 + t^{-6}.
\end{align*}
To complete the estimate of $g_{3,4}$, we only need observe that
by \eqref{eq:Mm} and \eqref{eq:BS},
\[
\biggl| \int\eta (\rho-t)\bM_\bX\biggr|
\lesssim \int_{|x|>Lt} |\eta| \rho |\bM_\bX|
\lesssim \|\eta\|_\rho \left(\int_{|x|>Lt} \rho\bM_\bX^2\right)^\frac 12
\lesssim t^{-\frac{13}2+4\delta}\lesssim t^{-6}.
\]
Lastly, we estimate $g_{3,5}$.
Using the cancellation
\[
\Delta M_k + f'(W_k)M_k - (\bl_k\cdot\nabla)^2 M_k=0,
\]
we see that
\[
g_{3,5}=
-2 \sum_k \int \ep\rho (f'(\bW)-f'(W_k))M_k +
2\sum_k\int \ep( (\bC-\rho\bl_k )\cdot\nabla )(\bl_k\cdot\nabla M_k).
\]
We treat the first term in the above expression of $g_{3,5}$.
For $x\in \Omega=\RR^5\setminus \cup_k B_k$, using \eqref{eq:Av},
we have
\[
|f'(\bW)-f'(W_k)|\lesssim \sum_k |W_k|^\frac43+\sum_k |v_k|^\frac43
\lesssim (t^{-2}q^\frac32)^\frac43\lesssim t^{-3}.
\]
Moreover, for $x\in B_k$, we have
\[
|f'(\bW)-f'(W_k)|\lesssim |W_k|^\frac13 \biggl(\sum_{m\neq k} |W_m|+\sum_m |v_m|\biggr)
\lesssim q_k (t^{-2}q^\frac32)\lesssim t^{-2}.
\]
Thus, using \eqref{eq:pC} 
\[
\biggl|\int \ep\rho (f'(\bW)-f'(W_k))M_k\biggr|
\lesssim t^{-1}\left(t^{-\frac12}\cN+t^{-4}\right)\int |\ep| q^3
\lesssim t^{-6}.
\]
For the second term in the
above expression of $g_{3,5}$, using \eqref{eq:en} and \eqref{eq:pC}, we have
\begin{align*}
\biggl|\int \ep( (\bC-\bl_k t)\cdot\nabla )(\bl_k\cdot\nabla M_k)\biggr|
&\lesssim \left(t^{-\frac12}\cN+t^{-4}\right)\|q_k\ep\| \|(\bC-\bl_k\rho)q_k^4\|\\
&\lesssim  \frac{\alpha^\frac1{10}}t \int \rho(|\nabla \ep|^2 + \eta^2) + t^{-6}.
\end{align*}
Moreover,
\[
 \int_{|x|\geq L t} |\ep| (\rho-t) |\nabla(\bl_k\cdot\nabla M_k| 
\lesssim t^{-3+2\delta} \|q \ep\|_{L^2} \biggl(\int_{|x|\geq Lt} \rho q^8\biggr)^\frac12
\lesssim t^{-7+4\delta}\lesssim t^{-6}.
\]
The proof of the estimate of $g_3$ is complete. 

\medskip

\emph{Estimate of $g_4$.}
We claim
\begin{equation}\label{eq:g4}
|g_4|\lesssim_\delta t^{-6+3\delta}.
\end{equation}
Recall the expression of $g_4$, after one integration by parts,
as for $g_3$, 
\begin{align*}
g_4 & = 2 \int \rho\nabla\ep\cdot\nabla \bR_\bW
-2 \int \ep f'(\bW) \rho\bR_\bW +2 \int (\bC \cdot \nabla \ep) \bR_{\bX}
+ 4 \int \bR _{\bX} \ep \Phi\\
&\quad + 2 \int \eta \left(\rho \bR_\bX + \bC\cdot \nabla \bR_\bW
+ 2 \Phi \bR_\bW \right) .
\end{align*}
Using the Cauchy-Schwarz inequality on each of the terms in $g_4$, 
$|f'(\bW)|\lesssim q^2$,
and the Hardy inequality $\int\rho|x|^{-2}f^2\lesssim \int \rho|\nabla f|^2$ taken from 
\eqref{eq:e1}, we find
\[
|g_4|
\lesssim(\|\nabla\ep\|_\rho+\|\eta\|_\rho)(\|\nabla\bR_\bW\|_\rho+\|\bR_\bX\|_\rho).
\]
Using \eqref{eq:BS}, \eqref{eq:pC} and \eqref{eq:Rr}, we conclude for this term
\[
|g_4|\lesssim_\delta (\|\nabla\ep\|_\rho+\|\eta\|_\rho)\left(t^{-\frac72+\delta}+t^{-\frac32}\cD \right)
\lesssim_\delta  t^{-6+3\delta}.
\]

\medskip

\emph{Estimate of $g_5$.}
We claim
\begin{equation}\label{eq:g5}
|g_5|\lesssim \frac\alpha  t\int \rho (|\nabla \ep|^2+\eta^2).
\end{equation}
To prove the estimate \eqref{eq:h5}, we decompose 
$g_5 = \sum_{n=1}^6 g_{5,n}$ where
\begin{align*}
g_{5,1} & = 
-4 \int \ep (\nabla \ep\cdot \nabla\Phi)\\
g_{5,2} & = 
2 \int \Bigl(\Bigl(\partial_t \bC - \sum_k \bl_k \qk\Bigr)\cdot \nabla \ep\Bigr) \eta\\
g_{5,3}& = 4 \int (\partial_t\Phi) \ep \eta\\
g_{5,4}& = \int_{|x|<Lt} \Bigl(1-\Phi-\sum_k\qk\Bigr)\left(|\nabla\ep|^2+\eta^2\right)\\
g_{5,5} & = 
\int |\nabla\ep|^2\tr G - 2 \int (\nabla\ep)^\TR G\, \nabla\ep
-\int \eta^2 \tr G\\
g_{5,6}& = -10\int (F(\bW+\ep)-F(\bW)-f(\bW) \ep) \Phi
+ 4 \int (f(\bW+\ep)-f(\bW)) \ep \Phi\\
&\quad -2 \int \left( F(\bW + \ep) - F(\bW)-f(\bW) \ep \right) \tr G\\
 & \quad + \frac{\alpha}t \int_{|x|>Lt} \rho\left(|\nabla\ep|^2+\eta^2- 2 (F (\bW +\ep )-F (\bW )-f (\bW )\ep ) \right).
\end{align*}
First, by \eqref{eq:Ph}, $|\nabla \Phi|\lesssim \alpha \Theta q
\lesssim \alpha t^{-1} \rho q$.
Thus, using \eqref{eq:e1}, 
\[
|g_{5,1}|
\lesssim \frac\alpha t \left(\int \rho q^2 \ep^2+\int \rho |\nabla\ep|^2\right)
\lesssim \frac \alpha t \int \rho |\nabla \ep|^2.
\]
Second, by \eqref{eq:GG}, $|\partial_t \bC -\sum_k \bl_k \qk | \lesssim \alpha t^{-1}\rho$
and so
\[
|g_{5,2}|\lesssim \frac\alpha t \int \rho(|\nabla \ep|^2 + \eta^2).
\]
Third, using \eqref{eq:Ph}, we have
\[
|\partial_t \Theta|\lesssim 
\alpha\Theta( t^{-1}+q)
\lesssim \frac{\alpha}t\rho (|x|^{-1} + q).
\]
Thus, by the Cauchy-Schwarz inequality and \eqref{eq:e1},
\[
|g_{5,3}|
\lesssim \frac \alpha t \int \rho\left((|x|^{-2}+q^2) \ep^2 + \eta^2\right)
\lesssim \frac \alpha t \int \rho\left( |\nabla \ep|^2 + \eta^2\right).
\]
Fourth, since $|1-\psi-\sum_k\qk|\lesssim t^{-\alpha}$, it holds
\[
|g_{5,4}|\lesssim t^{-\alpha} \int |\nabla\ep|^2+\eta^2\lesssim 
t^{-1-\alpha}\int \rho\left(|\nabla\ep|^2+\eta^2\right).
\]
Then, by \eqref{eq:G2},
$|G|\lesssim \alpha\Theta\lesssim \alpha\rho/t$, and so
\begin{equation*}
|g_{5,5}|\lesssim\biggl| \int |\nabla\ep|^2\tr G \biggr|+
\biggl| \int (\nabla\ep)^\TR G\, \nabla\ep \biggr|+ 
\left| \int \eta^2 \tr G\right|
\lesssim \frac\alpha  t\int \rho (|\nabla \ep|^2+\eta^2).
\end{equation*}
Lastly, we have
\[
|F(\bW+\ep)-F(\bW)-f(\bW) \ep|
+|(f(\bW+\ep)-f(\bW))\ep| 
\lesssim |\ep|^{\frac{10}3} + q^4\ep^2.
\]
Moreover, 
we have already observed in the proof of \eqref{eq:cl} that
$\psi\leq 1-\qk\lesssim \alpha |x-\bl_k t-\by_k|^2$, for
any $k$. Thus,
$\Phi q^4= \Theta \psi q^4 \lesssim \alpha q^2$.
Thus, using the Sobolev inequality and the Hardy inequality,
and \eqref{eq:en}, for
$t$ sufficiently large,
\begin{align*}
&\left| \int (F(\bW+\ep)-F(\bW)-f(\bW) \ep) \Phi\right|
+ \left|\int (f(\bW+\ep)-f(\bW)) \ep \Phi\right|
\\ &\quad\lesssim \int |\ep|^{\frac{10}3} 
+ \alpha \int \ep^2 q^2
\lesssim \left( \int |\nabla \ep|^2\right)^\frac53+\alpha \int |\nabla \ep|^2
\lesssim \frac\alpha t \int \rho|\nabla \ep|^2.
\end{align*}
By \eqref{eq:GG}, $|G|\lesssim \alpha$, and so
using also $\int |\nabla \ep|^2\lesssim 1$ by \eqref{eq:en},
\[
\left|\int \left( F(\bW + \ep) - F(\bW)-f(\bW) \ep \right) \tr G\right|
\lesssim \alpha \int (|\ep|^{\frac{10}3} + q^4\ep^2)
\lesssim \alpha \int |\nabla \ep|^2
\lesssim \frac\alpha t\int \rho |\nabla\ep|^2.
\]
For the last term of $g_{5,6}$, proceeding as in the proof of 
\eqref{eq:h6}, one has
\begin{align*}
\frac{\alpha}t 
\biggl|\int_{|x|>Lt}\rho\left(F(\bW +\ep)-F(\bW)-f(\bW)\ep)\right)\biggr|
\lesssim \frac\alpha t\int \rho \big(q^4 \ep^2 + |\ep|^{\frac{10}3}\big)
\lesssim  \frac\alpha t\int \rho |\nabla\ep|^2.
\end{align*}

In conclusion of all the above estimates, we have obtained
\[
\cH' +\frac 2 {L^2}\frac \cH t \gtrsim_\delta -t^{-6+3\delta}
-\sum_k \left((z_k^-)^2+ (z_k^+)^2\right)
\]
by taking $\alpha>0$ small enough,
 $0<\delta\leq \alpha/4$ and $T>0$ sufficiently large depending on $\alpha$ and
$\delta$.
Now, $\alpha$ is fixed and $\delta=\alpha/4$ is fixed and $T>0$
is fixed.
Inserting the estimate
$(z_k^-)^2+ (z_k^+)^2 \lesssim t^{-7}$ from \eqref{eq:BS},
we obtain the result.
\end{proof}

\section{Existence}
To complete the proof of Theorem \ref{th:un} concerning the existence of multi-solitons,
we closely follow \cite[section 5]{MaM2}.
We will prove the following proposition, which is identical to \cite[Proposition 5.1]{MaM2},
but without restriction on the parameters.
The functions $(\bW,\bX)$ are defined in section \ref{S:2.2}.

\begin{proposition}\label{pr:S1}
Let $0<\delta\ll 1$.
There exist $T_0>0$
and a solution $u(t)$ of~\eqref{eq:NW} on $[T_0,+\infty)$ satisfying, for all $t\in [T_0,+\infty)$,
\begin{equation}\label{pr:S11}
\left\|\nabla u(t) - \nabla \bW(t)\right\|_{L^2}
+\left\| \partial_t u(t) - \bX(t) \right\|_{L^2} \lesssim t^{-3+2\delta}
\end{equation}
where $\lambda_k(t)$, $\by_k(t)$ are such that, for all $t\in [T_0,+\infty)$,
\begin{equation}\label{blay}
|\lambda_k(t)-\lambda_k^\infty|+|\by_k(t)-\by_k^\infty|\lesssim t^{-1}.
\end{equation}
\end{proposition}
Note that Proposition \ref{pr:S1} clearly implies Theorem \ref{th:un}.
Moreover, Proposition \ref{pr:S1} will be a key ingredient to prove Corollary \ref{th:dx}.
The rest of this section is devoted to the proof of Proposition~\ref{pr:S1}.

Let $S_n\to +\infty$.
Let $\zeta_{k,n}^{\pm}\in \RR$, $\nu_{k,n}^\Lambda\in\RR$  and $\nu_{k,n}^\nabla\in\RR^5$ small to be determined later.
For any large $n$, we consider the solution $u_n$ of
\begin{equation}\label{defun}
\left\{
\begin{aligned} 
\partial_t^2 u_n - \Delta u_n - |u_n|^{\frac 4{3}} u_n &= 0 \\ 
 (u_n(S_n),\partial_t u_n(S_n))^\TR & = \sum_k \bigg( \vec W_{k}^\infty (S_n) + c_k \vec V_k^\infty(S_n)
 +\sum_\pm \zeta_{k,n}^\pm ( \vec \theta_k^\infty \vec Z_{\bl}^\pm)(S_n)\\
&\quad
+ (\nu_{k,n}^\Lambda \Lambda_k W_k(S_n),0)^\TR + (\nu_{k,n}^\nabla \cdot \nabla W_k(S_n),0)^\TR \bigg).
\end{aligned}
\right.
\end{equation}
Since $(u_n(S_n),\partial_t u_n(S_n))\in \dot H^1\times L^2$, the solution $\vec u_n$ is well-defined in $\dot H^1\times L^2$ at least on a small interval of time around $S_n$.

We state uniform estimates on $u_n$, backwards in time up to some uniform $T_0\gg1$.
\begin{proposition}\label{pr:s4} There exist $n_0>0$ and $T_0>0$ such that, for any $n\geq n_0$, there exist
$(\zeta_{k,n}^{\pm})_k\in \RR^K \times \RR^K$, $(\nu_{k,n}^\Lambda)_k\in \RR^K$, $(\nu_{k,n}^\nabla)_k\in (\RR^5)^K$
with
\[
\sum_k |\zeta_{k,n}^{\pm}|^2 +\sum_k |\nu_{k,n}^\Lambda|^2+\sum_k |\nu_{k,n}^\nabla|^2 \lesssim {S_n^{-7}},
\]
and such that the solution $u_n$ of~\eqref{defun} is well-defined in $\dot H^1\times L^2$ on the time interval $[T_0,S_n]$
and satisfies, for all $t\in [T_0,S_n]$,
\begin{equation}\label{eq:un} 
\left\|\vec u_n(t) -\vec \bW_{n}(t)\right\|_{\dot H^1\times L^2} \lesssim t^{-3+2\delta}
\end{equation}
where $
\vec \bW_{n}(t,x)=
\vec \bW\left(t,x;\{\lambda_{k,n}(t)\},\{\by_{k,n}(t)\}\right)$ is defined in section \ref{S:2.2}
and
\begin{equation}\label{eq:deux}
|\lambda_{k,n}(t)-\lambda_k^\infty|+|\by_{k,n}(t)-\by_{k}^\infty|\lesssim t^{-1},\quad
\big| \dot \lambda_{k,n} (t)\big|+\big|\dot {\by}_{k,n} (t)\big|\lesssim t^{-2}.
\end{equation}
Moreover, 
\begin{itemize}
\item Uniform localization property:
\begin{equation}\label{eq:ul}
\mbox{For all $\epsilon>0$, there exists $A>0$ such that for all $n\geq n_0$, 
$\|u_n(T_0)\|_{(\dot H^1\times L^2)(|x|>A)} < \epsilon$.}
\end{equation}
\item Higher regularity property:
\begin{equation}\label{eq:hr}
\mbox{$\vec u_n\in \cC([T_0,S_n],\dot H^2\times \dot H^1)$ and for all $t\in [T_0,S_n]$,
$\|\vec u_n(t)\|_{\dot H^2\times \dot H^1} \lesssim 1$.}
\end{equation}
\end{itemize}
\end{proposition}

\subsection{Bootstrap setting}
We denote by $\cB_{\RR^K}(r)$ (respectively, $\cS_{\RR^K}(r)$) the open ball (respectively, the sphere) of $\RR^K$ of center $0$ and of radius $r>0$, for 
the norm $|(\xi_k)_k|= (\sum_{k} \xi_k^2 )^{1/2}$. 

For $t=S_n$ and for $t<S_n$ as long as $u_n(t)$ is well-defined in $\dot H^1\times L^2$ and satisfies~\eqref{hyp:4}, we will consider the decomposition of $\vec u_n(t)$ from Lemma~\ref{le:dc}. For simplicity of notation, we will denote the parameters $\lambda_{k,n}$, $\by_{k,n}$ and $\vec\varepsilon_n$ of this decomposition by $\lambda_k$, $\by_k$ and $\vec\varepsilon$.

We start with a technical result similar to~\cite[Lemma 3.1]{Je19} and \cite[Lemma 4.1]{CMYZ}. 
This result will allow us to adjust the data at $t=S_n$ in~\eqref{defun}.

\begin{lemma}[Choosing the data] \label{le:modu2}
There exist $n_0>0$ and $C>0$ such that, for all $n\geq n_0$, for any
$(\xi_k)_k\in \bar{\cB}_{\RR^K}(S_n^{-7/2})$, there exists
unique $(\zeta_{k,n}^{\pm})_k\in \cB_{\RR^K}(C S_n^{-7/2})$,
$(\nu_{k,n}^\Lambda)_k\in \mathcal{B}_{\RR^K}(C S_n^{-7/2})$
and $(\nu_{k,n}^\nabla)_k\in \mathcal{B}_{(\RR^5)^K}(C S_n^{-7/2})$ such that
the decomposition of $u_n(S_n)$ satisfies the estimates
\begin{equation}\label{modu:2}
z_k^-(S_n)=\xi_k,\quad
z_k^+(S_n)=0,
\end{equation}
\begin{equation}\label{modu3}
|\lambda_k(S_n)-\lambda_k^\infty|+
|\by_k(S_n)-\by_k^\infty|
+ \left\|\vec \varepsilon(S_n)\right\|_{\dot H^1\times L^2}
+\left\|\vec \varepsilon(S_n)\right\|_{\dot H^2\times \dot H^1}\lesssim S_n^{-7/2},
\end{equation}
and the orthogonality conditions \eqref{eq:or}.
\end{lemma}
From now on, for any $ (\xi_{k})_k\in \bar{\cB}_{\RR^K}(S_n^{-7/2})$,
we fix $(\zeta_{k,n}^{\pm})_k$, $(\nu_{k,n}^\Lambda)_k$, $(\nu_{k,n}^\nabla)_k$
 as given by Claim~\ref{le:modu2} and the corresponding solution $u_n$ of~\eqref{defun}.
We fix $0<\delta<\frac 1{10}$.
\smallbreak

The proof of Proposition~\ref{pr:s4} is based on the bootstrap estimates \eqref{eq:BS}.
Set
\begin{equation}\label{def:tstar}
T^*=T_n^*((\xi_{k})_k)=\inf\{t\in[T_0,S_n]\ ; \ \mbox{$u_n$ satisfies~\eqref{hyp:4} and \eqref{eq:BS} holds on $[t,S_n]$} \} .
\end{equation}
Note that by Claim~\ref{le:modu2}, estimate~\eqref{eq:BS} is satisfied 
at $t=S_n$.
Moreover, if~\eqref{eq:BS} is satisfied
on $[\tau,S_n]$ for some $\tau\leq S_n$ then by the well-posedness theory 
and continuity, 
$u_n(t)$ is well-defined and satisfies the decomposition of Lemma~\ref{le:dc} on $[\tau',S_n]$, 
 for some $\tau'<\tau$. 
In particular, the definition of $T^*$ makes sense. In what follows, we will prove that there exists $T_0$ large enough and at least one choice of $(\xi_{k})_{k}\in \cB_{\RR^K}(S_n^{-7/2})$ so that $T^*=T_0$, which is enough to finish the proof of Proposition~\ref{pr:s4}.

\subsection{Closing the bootstrap estimates}\label{s:4.2}
In sections \ref{s:4.2}--\ref{s:4.4}, we work on a solution $u_n$
for a given $n$ and we drop the index $n$ for the simplicity of notation.
Observe that  for the parameters $\lambda_k$, $\by_k$ and $z_k^\pm$, the bootstrap  estimates \eqref{eq:BS} are exactly the same as the ones in \cite[(5.8)]{MaM2}, and that the bootstrap estimate \eqref{eq:BS}  for the function
$\vec \varepsilon$ is stronger in the sense that \eqref{eq:BS} implies \eqref{eq:en}, 
which is exactly the estimate of $\vec \varepsilon$ in \cite[(5.8)]{MaM2}.
Therefore, except for the energy estimates on $\vec\varepsilon$, we will reproduce exactly the same proofs as in \cite[section~5.5]{MaM2}.

We start by closing estimates on all the parameters except the (backward) unstable modes.

\begin{lemma}[Closing estimates for $\lambda_k$, $\by_k$ and $z_k^+$]\label{le:bs1}
For $C_0>0$ large enough, for all $t\in[T^*,S_n]$,
\begin{equation}\label{eq:BSd}\left\{\begin{aligned}
& |\lambda_k(t)-\lambda_k^\infty|\leq \frac {C_0} 2t^{-1},\quad |\by_k(t)-\by_k^\infty|\leq \frac {C_0}2t^{-1},\\
&|z_k^+(t)|^2\leq \frac 12 t^{-7}.
\end{aligned}\right. \end{equation}
\end{lemma}
\begin{proof}
We reproduce the proof for the reader convenience.
From~\eqref{eq:pC} and~\eqref{eq:BS}, we have
$\left|\frac {\dot \lambda_k}{\lambda_k}\right|+|\dot {\by}_k|\leq C t^{-2}$
where the constant $C$ depends on the parameters of the two solitons, but not on the bootstrap constant~$C_0$. By integration on $[t,S_n]$ for 
$T^*\leq t\leq S_n$, and~\eqref{modu3}, we obtain
\[
|\lambda_k(t)-\lambda_k^\infty|\leq |\lambda_k(t)-\lambda_k(S_n)|+|\lambda_k(S_n)-\lambda_k^\infty|\leq C' t^{-1},
\]
and similarly, $|\by_k(t)-\by_k^{\infty}|\leq C't^{-1}$, 
where $C'$ is also independent of $C_0$. We choose $C_0=2 C'$.
Now, we prove the bound on $z_k^+(t)$. Let $\beta_k=\frac{\sqrt{\lambda_0}}{\lambda_k^{\infty}}(1-|\bl_k|^2)^{1/2}>0$.
Then, from~\eqref{eq:zC} and~\eqref{eq:BS}, 
\[
\frac d{dt} \left( e^{-\beta_k t} z_k^+\right)\lesssim e^{-\beta_k t} t^{-4+2\delta}. 
\]
Integrating on $[t,S_n]$ and using~\eqref{modu:2}, we obtain
$- z_k^+(t) \lesssim t^{-4+2\delta}$.
Doing the same for $-e^{-\beta_k t} z_k^+$, we obtain the conclusion for $T_0$ large enough.
\end{proof}

Now, we strictly improve the bootstrap estimate for $\vec\varepsilon$
using the norm \eqref{eq:NN} and the functional $\cH$ defined in section~\ref{S:3}.

\begin{lemma}[Closing estimates for $\vec \varepsilon$]\label{le:bs1bis}
For $C_0>0$ large enough, for all $t\in[T^*,S_n]$,
\begin{equation}\label{eq:BSe}
\cN(t)\leq \frac 12 t^{-\frac 52 +2\delta}.
\end{equation}
\end{lemma}
\begin{proof}
Integrating the differential inequality from Lemma \ref{le:en},
\[
\frac d{dt}\left( t^{\frac2{L^2}} \cH \right) \geq - C_\delta t^{\frac2{L^2}-6+3\delta}
\]
on $[t,T_n]$, we obtain using $\cH(t_n)=0$,
\[
\cH(t) \lesssim_\delta   t^{-5+3\delta}.
\]
Using the coercivity Lemma \ref{le:co} and $(z_k^-)^2+ (z_k^+)^2 \lesssim t^{-7}$ from \eqref{eq:BS},
we obtain
\[
\|\nabla\ep\|_\rho+\|\eta\|_\rho\lesssim \alpha^\frac1{10} t^{-\frac52+2 \delta},
\]
which implies \eqref{eq:BSe} for $\alpha$ small enough.
\end{proof}

Lastly, we control the (backward) unstable parameters $z_k^-$ using a topological argument.
The proof is identical to \cite[Lemma 5.6]{MaM2}.

\begin{lemma}[Control of unstable directions]\label{le:bs2}
There exist 
$(\xi_{k,n})_{k}\in \cB_{\RR^K}(S_n^{-7/2})$ such that, for $C^*>0$ large enough, $T^*((\xi_{k,n})_k)=T_0$.
\end{lemma}
\begin{proof}
We follow the strategy initiated in \cite[Lemma 6]{CoMM} to control unstable direction in such multi-soliton constructions.
The proof is by contradiction, we assume that for any $(\xi_k)_k\in \bar{\cB}_{\RR^K}(S_n^{-7/2})$,
$T^*((\xi_k)_k)$ defined by~\eqref{def:tstar} satifies $T^*\in (T_0,S_n]$. In this case, by Lemma~\ref{le:bs1} and continuity,
 it holds necessarily
\[
\sum|z_{k}^-(T^*)|^2= \frac{1}{(T^*)^{7}}.
\]
Let $\bar \beta = \min_k \beta_k$.
From~\eqref{eq:zC} and~\eqref{eq:BS}, for all $t\in [T^*,S_n]$, one has
\[
\frac{d}{dt}\left(t^7 \left(z_{k}^- \right)^2\right) =
2 t^7 z_{k}^- \frac{d}{dt}z_{k}^-
+ 7 t^6 \left(z_{k}^- \right)^2 \leq 
- 2 t^7 \beta_k 
\left(z_{k}^- \right)^2 + \frac{C }{t} \leq -2 \bar \beta t^7 \left(z_{k}^- \right)^2 + \frac{C }{t}.
\]
Thus, for $T_0$ large enough, and any $t_0\in [T^*,S_n]$,
\[
\sum|z_{k}^-(t_0)|^2= \frac{1}{t_0^{7}}\quad \mbox{implies}\quad 
\left. \frac{d }{dt} \left(t^{7} \sum|z_{k}^-(t)|^2\right)\right|_{t=t_0} \leq 
- 2\bar \beta 
+ \frac{C}{T_0} \leq- \bar \beta. 
\]
As a standard consequence of this transversality property, the maps
\[
(\xi_k)_{k}\in \bar{\cB}_{\RR^K}(S_n^{-7/2}) \mapsto T^*((\xi_k)_k)
\]
and 
\[
(\xi_k)_k\in \bar{\cB}_{\RR^K}( S_n^{-7/2} ) \mapsto 
\cM((\xi_k)_k)= \left(\frac {T^*}{S_n}\right)^{7/2} (z_k^-(T^*))_k\in \cS_{\RR^K}( S_n^{-7/2} )\]
are continuous.
Moreover, $\cM$ restricted to $\cS_{\RR^K}( S_n^{-7/2} )$ is the identity
and this is contradictory with Brouwer's fixed point theorem. 
\smallbreak

Estimates~\eqref{eq:un} follow directly from the estimates~\eqref{eq:BS} on $\varepsilon(t)$, $\lambda_k(t)$, $\by_k(t)$.
\end{proof}

\subsection{Proof of the uniform localization property}
From the bootstrap estimate \eqref{eq:BS}, which we know is satisfied at $t=T_0$,
we know that uniformly in $n\geq n_0$, it holds
\[
\int \left(|\nabla\varepsilon(T_0)|^2+\eta^2(T_0)|\right) \rho(T_0,x) dx =
\|\nabla \varepsilon(T_0)\|_\rho^2+\|\eta(T_0)\|_\rho^2\lesssim 1
\]
where by definition the weight function $\rho(T_0)$ satisfies for $|x|>LT_0$,
\[
\rho(T_0,x)=T_0^\alpha L^{1-\alpha} |x|^{1-\alpha},
\]
where $\alpha>0$ is small. Thus, for any $A>LT_0$, it holds
\[
\int_{|x|>A} \left(|\nabla\varepsilon(T_0)|^2+\eta^2(T_0)|\right) dx
\lesssim A^{\alpha-1} \int \left(|\nabla\varepsilon(T_0)|^2+\eta^2(T_0)|\right) \rho(T_0,x) dx
\lesssim A^{\alpha-1},
\]
which implies the uniform estimate \eqref{eq:ul}.

Here, the uniform localization property is obtained as a simple consequence of the uniform bound on the norm $\cN$. 
This is a main difference with \cite{MaM2}, where such a property had to be proved separately 
in \cite[section 5.5]{MaM2}.

\subsection{Proof of the regularity property}\label{s:4.4}

The proof of the uniform bound in $\dot H^2\times \dot H^1$ is omitted since it is identical
to the one in \cite[section 5.4]{MaM2}.

\subsection{End of the proof of existence} 
From the estimates of Proposition~\ref{pr:s4} on the sequence of initial data 
$(\vec u_n(T_0))$, it follows that 
up to the extraction of a subsequence (still denoted by $(\vec u_n)$), the sequence $(\vec u_n(T_0))$ converges to some $(u_0,u_1)^\TR$
in $\dot H^1\times L^2$ as $n\to +\infty$.
Consider the solution $u(t)$ of~\eqref{eq:NW} associated to the initial data $(u_0,u_1)^\TR$ at $t=T_0$.
Then, by the continuous dependence of the solution of~\eqref{eq:NW} with respect to its initial data in the energy space $\dot H^1\times L^2$ 
(see \emph{e.g.}~\cite{KeMe} and references therein) and the uniform bounds~\eqref{eq:un}, the solution $u$ is well-defined in the energy space on $[T_0,\infty)$.

Recall that we denote by $\lambda_{k,n}$ and $\by_{k,n}$ the parameters of the decomposition of $u_n$ on $[T_0,S_n]$.
By the uniform estimates in~\eqref{eq:deux}, using Ascoli's theorem and a diagonal argument, it follows that there exist continuous functions $\lambda_k$ and $\by_k$ such that up to the extraction of a subsequence, $\lambda_{k,n}\to \lambda_k$, $\by_{k,n}\to \by_k$ uniformly on compact sets of $[T_0,+\infty)$, and on $[T_0,+\infty)$,
\[
|\lambda_k(t)-\lambda_k^\infty|\lesssim t^{-1},\quad |\by_k(t)-\by_k^\infty|\lesssim t^{-1}.
\]
Passing to the limit in~\eqref{eq:un} for any $t\in [T_0,+\infty)$, 
we finish the proof of Proposition~\ref{pr:S1}.

\section{Inelasticity}\label{S:5}
The proof of the inelasticity result Corollary \ref{th:dx}, based on the refined
existence result Proposition \ref{pr:S1} and arguments from \cite{MaM2} is a direct extension of
\cite[section 6]{MaM2}. Therefore, we will only sketch the proof.
Our goal is to prove that the solution obtained in Proposition \ref{pr:S1} (observe that the statement of Proposition \ref{pr:S1} is more precise than the one of Theorem \ref{th:un})
satisfies \eqref{eq:di}.

The first step is to justify a remark after Corollary \ref{th:dx}
saying that it makes sense to consider the solution
$u(t,x)$, for any time, but outside a certain wave cone. Indeed, let 
\[
R\gg1 \quad\mbox{and}\quad t_R=R^\frac{11}{12}.
\]
Let $\chi_1:\RR^5\to \RR$ be a smooth, radially symmetric function such that $\chi_1\equiv 1$ for $|y|>1$
and $\chi_1\equiv 0$ for $|y|<\frac 12$. Let $\chi_R(x)= \chi_1(x/R)$. Define $\vec u_R=(u_R,\partial_t u_R)^\TR$ the solution of~\eqref{eq:NW} with the following data at the time $t_R$
\[
u_R(t_R) = u(t_R) \chi_R,\quad \partial_t u_R(t_R) = \partial_t u(t_R) \chi_R.
\]
Then, from~\cite[Claim 6.1]{MaM2}, it holds, for $R>0$ sufficiently large,
\begin{equation}\label{eq:61}
\|\vec u_R(t_R) \|_{\dot H^1\times L^2} \lesssim R^{-\frac 32}.
\end{equation}
We omit the proof of \eqref{eq:61}, based on \eqref{eq:bW} and \eqref{pr:S11},
and refer to \cite{MaM2}.
By the small data Cauchy theory (\cite{KeMe}), for $R$ large enough, 
the solution $\vec u_R$ is global and bounded in $\dot H^1\times L^2$.
Moreover, since
$\vec u_R(t_R) = \vec u(t_R)$, for $|x|>R$,
by the property of finite speed of propagation of the wave equation, 
we can define globally $\vec u(t,x)$ on $\Sigma_R$ by setting $u(t,x)=u_R(t,x)$. 
The point is that this extension makes sense even if $u(t)$ is not global in $\dot H^1\times L^2$ for negative times.
Our goal in the rest of this section is to prove the following statement, for $R$ large,
\begin{equation}\label{reduction}
\liminf_{t\to -\infty} \|\nabla u(t)\|_{L^2(|x|>R+|t-t_R|)} \gtrsim R^{-\frac 52} ,
\end{equation}
which implies, for $A=R+t_R$ large enough,
\[
\liminf_{t\to -\infty} \|\nabla u(t)\|_{L^2(|x|>|t|+A)} \gtrsim A^{-\frac 52}.
\]

The second step of the proof of Corollary \ref{th:dx} is to show that it is sufficient to consider
the linear wave equation instead of the nonlinear equation \eqref{eq:NW}.
Indeed, define $\vec u_{\rm L} = (u_{\rm L},\partial_t u_{\rm L})^\TR$ the (global) solution of the $5$D linear wave equation with initial data at $t=t_R$,
\begin{equation}\label{uL}\left\{\begin{aligned}
& \partial_t^2 u_{\rm L} - \Delta u_{\rm L} = 0 \quad \mbox{on $\RR\times \RR^5$},\\
& u_{\rm L}(t_R)=u_R(t_R) = u(t_R) \chi_R,\quad \partial_t u_{\rm L}(t_R)=\partial_t u_R(t_R) = \partial_t u(t_R) \chi_R \quad
\mbox{on $\RR^5$.}
\end{aligned}\right.\end{equation}
Using \eqref{eq:61} and \cite[Proposition~2.12]{MaM2}, it follows that for $R$ large enough, 
\begin{equation}\label{reduc_lin}
\sup_{t\in \RR} \|\vec u_{\rm L}-\vec u_{R}\|_{\dot H^1\times L^2} \lesssim R^{-\frac 73 \cdot \frac 32} = R^{-\frac 72}.
\end{equation}
To prove~\eqref{reduction}, it is thus sufficient to obtain the same statement for $\vec u_{\rm L}$.

For future use, we also define
\[
w_\beta(t,x)=\frac \epsilon{\lambda^{\frac 32}} W_\bl\left(\frac{x-\bl t-\by}{\lambda} \right), \quad \vec w_\beta= \left(\begin{array}{c} w_\beta \\ \partial_t w_\beta \end{array}\right).
\]
where $\beta=(\bl,\lambda,\by,\epsilon)$,
and $\vec w_{\beta,R}=(w_{\beta,R},\partial_t w_{\bl,R})^\TR$ the solution of~\eqref{eq:NW} with truncated data at $t_R$
\[
w_{\beta,R}(t_R) = w_\beta(t_R) \chi_R,\quad \partial_t w_{\beta,R}(t_R) = \partial_t w_\beta(t_R) \chi_R.
\]
Moreover, we define
$\vec w_{\beta,\rm L} = (w_{\beta,\rm L},\partial_t w_{\beta,\rm L})^\TR$ the solution of the $5$D linear wave equation with data at $t_R$
\[
 w_{\beta,\rm L}(t_R)= w_{\beta,R}(t_R) = w_\beta(t_R) \chi_R,\quad \partial_t w_{\beta,\rm L}(t_R)=\partial_t w_{\beta,R}(t_R) = \partial_t w_\beta(t_R) \chi_R.
\]
Lastly, we define
\[
\beta_k = (\bl_k,\lambda_k(t_T),\by_k(t_R),\epsilon_k).
\]

The third ingredient, which is the key argument in the inelasticity proof, is the theory of channel of energy for radial solutions of the linear wave equation in $\RR^5$.
We state here the specific result from~\cite{KeLS} to be used in the proof
(see also~\cite{DKM1} and \cite{KLLS} for similar statements in any odd space dimension).
\begin{proposition}[\cite{KeLS}, Proposition~4.1]\label{pr:ch}
There exists a constant $C>0$ such that any radial energy solution $U_{\rm L}$ of the {\rm $5$D} linear wave equation
\[
\left\{ \begin{aligned}
&\partial_t^2 U_{\rm L} - \Delta U_{\rm L} = 0, \quad (t,x)\in (-\infty,\infty)\times \RR^5,\\
& {U_{\rm L}}_{|t=0} = U_0\in \dot H^1,\quad 
{\partial_t U_{\rm L}}_{|t=0} = U_1\in L^2,
\end{aligned}\right.
\]
satisfies, for any $R>0$, either
\[
\liminf_{t\to -\infty} \int_{|x|>|t|+R} |\partial_t U_{\rm L}(t,x)|^2+|\nabla U_{\rm L}(t,x)|^2 dx 
\geq C\| \pi^\perp_R (U_0,U_1) \|^2_{(\dot H^1\times L^2) (|x|>R)}
\]
or
\[
\liminf_{t\to +\infty} \int_{|x|>|t|+R} |\partial_t U_{\rm L}(t,x)|^2+|\nabla U_{\rm L}(t,x)|^2 dx 
\geq C \| \pi^\perp_R (U_0,U_1) \|^2_{(\dot H^1\times L^2) (|x|>R)}
\]
where $\pi^{\perp}_R(U_0,U_1)$ denotes the orthogonal projection of $(U_0,U_1)^\TR$ onto the complement of the plane
\[
 {\rm span} \left\{ (|x|^{-3},0)^\TR,(0,|x|^{-3})^\TR \right\}
\]
in $(\dot H^1\times L^2) (|x|>R)$.
\end{proposition}
To use such a result, we need to consider a radial solution of the linear wave equation, 
though of course, multi-solitons such as the solution $u$ of Proposition \ref{pr:S1},
or the restriction $u_L$  are by construction non-radial.

We define
\begin{equation}\label{defUL}
U_{\rm L}(t,x) = \fint_{|y|=|x|} \left( u_{\rm L} - \sum_k w_{\beta_k,\rm L}\right)(t,y) d\omega(y),\quad
\vec U_{\rm L} = \left(\begin{array}{c} U_{\rm L} \\ \partial_t U_{\rm L} \end{array}\right).
\end{equation}
It is well-known that $U_{\rm L}$ is a solution of the radial wave equation in $\RR^5$ and
our intention is to apply Proposition \ref{pr:ch} to $U_{\rm L}$.

Let
\[
\Psi 
= \sum_{k=1}^K (1-|\bl_k|^2)^{\frac 32}\lambda_k^\infty\Biggl( \sum_{m\neq k}\epsilon_{m}(\lambda_{m}^\infty)^{\frac 32}|\bs_{k,m}|^{-3} \Biggr)
\]
where for $k\neq m$, the value of $\bs_{k,m}$ is defined in \eqref{eq:bb}.

In the statement of Corollary \ref{th:dx}, we assume that $\Psi\neq 0$.
This allows us to obtain the following lower bound on the exterior energy for
$\vec U_L$.
\begin{lemma}\label{le:UL} 
Assume that $\Psi\neq 0$.
Then, for $R>0$ sufficiently large, it holds
\[
\|\pi_R^\perp \vec U_{\rm L}(t_R)\|_{(\dot H^1\times L^2)(|x|>R)} \gtrsim |\Psi| R^{-\frac52}.
\]
\end{lemma}
\begin{remark}
The lower bound in Lemma \ref{le:UL} is obtained thanks to the sharp computation of the
nonlinear interactions between the solitons.
Due to the addition of signed two by two contributions of all solitons to the nonlinear interaction,
there is a possibility of cancellation of the total contribution at
the main order, achieved if $\Psi=0$. In that case, we are not able to 
get information on $\|\pi_R^\perp \vec U_{\rm L}(t_R)\|_{(\dot H^1\times L^2)(|x|>R)}$
due to error terms.
Therefore, with this technique of proof, the condition $\Psi\neq 0$ is needed.
\end{remark}
\begin{proof}
We do not reproduce the proof of Lemma \ref{le:UL} since it is identical to the one of
\cite[Lemma 6.3]{MaM2}. We only recall that the lower bound on the projection of $\vec U_{\rm L}$
is not due to the solitons $W_k$ (since at the main order they have the exact decay $(|x|^{-3},0)$
which has a zero projection), but to the explicit asymptotic behavior of the function $v_\bl$ defined in Lemma~\ref{le:as} and used in the definition of the refined approximate solution $\vec \bW$
 in \eqref{eq:WW}.
Indeed, \cite[Lemmas 3.3 and 3.4]{MaM2} provide the main order of the asymptotic behavior of
$v_\bl$ of the form $v_\bl(t,x)\approx C_\bl |x|^{-4}$ for $1\ll t \ll |x|$,
for a non-zero constant $C_\bl$.
The radial function $|x|^{-4}$ has a nonzero projection $\pi_R^\perp$ with a size
in $\dot H^1(|x|>R)$ comparable to
\[
\biggl(\int_{|x|>R} |x|^{-10} dx\biggr)^\frac12 \approx C R^{-\frac 52}, 
\]
which justifies the lower bound in $R^{-\frac52}$.
\end{proof}
Using the method of channel of energy Proposition \ref{pr:ch} and the lower bound in Lemma \ref{le:UL},
we obtain that either
\[
\liminf_{t\to -\infty} \|\vec U_{\rm L}(t)\|_{(\dot H^1\times L^2)(|x|>R+|t-t_R|)} \gtrsim R^{-\frac 52}
\]
or
\[
\liminf_{t\to +\infty} \|\vec U_{\rm L}(t)\|_{(\dot H^1\times L^2)(|x|>R+|t-t_R|)} \gtrsim R^{-\frac 52} .
\]
By the definition of $\vec U_{\rm L}$ in~\eqref{defUL} and the decay properties of 
the solitons $W_k$ outside a large wave cone (at the main order solitons do not channel energy), we obtain that either
\[
\liminf_{t\to +\infty} \|\vec u_{\rm L}(t)\|_{(\dot H^1\times L^2)(|x|>R+|t-t_R|)} \gtrsim R^{-\frac 52} ,
\]
or
\[
\liminf_{t\to -\infty} \|\vec u_{\rm L}(t)\|_{(\dot H^1\times L^2)(|x|>R+|t-t_R|)}\\
\gtrsim R^{-\frac 52} .
\]
By~\eqref{reduc_lin} (\emph{i.e.}, comparing linear and nonlinear solutions), it follows that, for large $R$,
either 
\[
\liminf_{t\to +\infty} \|\vec u(t)\|_{(\dot H^1\times L^2)(|x|>R+|t-t_R|)} \gtrsim R^{-\frac 52} 
\quad \mbox{or}\quad
\liminf_{t\to -\infty} \|\vec u(t)\|_{(\dot H^1\times L^2)(|x|>R+|t-t_R|)} \gtrsim R^{-\frac 52} .
\]
However, the solution $\vec u(t)$ is a pure multi-soliton as $t\to +\infty$ and
thus, for any large $R$,
\[
\lim_{t\to +\infty} \|\vec u(t)\|_{(\dot H^1\times L^2)(|x|>R+|t-t_R|)} =0.
\]
Therefore, for large $R$,
\[
\liminf_{t\to -\infty} \|\nabla u(t)\|_{L^2(|x|>R+|t-t_R|)} +
 \|\partial_t u(t)\|_{L^2(|x|>R+|t-t_R|)} \gtrsim R^{-\frac 52} .
\]
By \cite[Proof of Proposition 4.1]{MaM2}, we obtain \eqref{reduction}.
For more details on the proof of inelasticity, we refer the reader to \cite[section 6]{MaM2}.

\end{document}